\documentclass[12pt,letterpaper]{article}

\pdfoutput=1

\usepackage{graphics,epsfig,graphicx,color}
\graphicspath{{/EPSF/}{../Figures/}{Figures/}}
\usepackage{subfig}
\newsubfloat{figure}
\usepackage{amssymb,latexsym}

\usepackage{setspace}

\usepackage{amsmath}

\usepackage{mathrsfs,amsfonts}

\usepackage{amsthm,amsxtra}

\newif\ifPDF
\ifx\pdfoutput\undefined
	\PDFfalse
\else
	\ifnum\pdfoutput > 0
		\PDFtrue
	\else
		\PDFfalse
	\fi
\fi

\ifPDF
	\usepackage{pdftricks}
	\begin{psinputs}
		\usepackage{pstricks}
		\usepackage{pstcol}
		\usepackage{pst-plot}
		\usepackage{pst-tree}
		\usepackage{pst-eps}
		\usepackage{multido}
		\usepackage{pst-node}
		\usepackage{pst-eps}
	\end{psinputs}
\else
	\usepackage{pstricks}
\fi

\ifPDF
	\usepackage[debug,pdftex,colorlinks=true, 
	linkcolor=blue, bookmarksopen=false,
	plainpages=false,pdfpagelabels]{hyperref}
\else
	\usepackage[dvips]{hyperref}
\fi

\usepackage{fancyhdr}

\newtheorem{theorem}{Theorem}[section]
\newtheorem{lemma}[theorem]{Lemma}

\newtheorem{proposition}[theorem]{Proposition} 
 
\newtheorem{corollary}[theorem]{Corollary}

\newcommand{\dsum}{\displaystyle\sum}

\newcommand{\bbR}{\mathbb R}

\newcommand{\be}{\mathbf e}

\newcommand{\E}{\epsilon} 
\newcommand{\T}{\tau}
\newcommand{\Tt}{\tau_2}
\newcommand{\Tk}{\tau_k}
\newcommand{\C}{\mathcal C}
\newcommand{\R}{\mathcal R}
\renewcommand{\S}{\mathcal S}
 
\newcommand{\eps}{\varepsilon}

\newcommand{\red}[1]{{\color{red} #1}}

\newcommand{\green}[1]{{\color{green} #1}}
\newcommand{\td}[1]{{\red{\textbf {[#1]}}}}

\newcommand{\old}[1]{}
\renewcommand{\be}{\begin{equation}}
\newcommand{\ee}{\end{equation}}

\title{Multipodal Structure {and Phase Transitions} in Large Constrained Graphs}

\author{
Richard Kenyon\thanks{Department of Mathematics, Brown University, Providence, RI 02912; rkenyon@math.brown.edu}
\and Charles Radin\thanks{Department of Mathematics, University of Texas, Austin, TX 78712; radin@math.utexas.edu}  
\and Kui Ren \thanks{Department of Mathematics, University of Texas, Austin, TX 78712; ren@math.utexas.edu} 
\and Lorenzo Sadun\thanks{Department of Mathematics, University of Texas, Austin, TX 78712; sadun@math.utexas.edu} 
}

\begin{document}
\maketitle

\begin{abstract}
We study the asymptotics of large, simple, labeled graphs
constrained by the densities of edges and 
of $k$-star subgraphs, $k\ge 2$ fixed.
We prove that under such constraints graphs
are ``multipodal'': asymptotically in the number of vertices 
there is a partition of the
vertices into $M < \infty$ subsets $V_1, V_2, \ldots, V_M$, and a set
of well-defined probabilities $g_{ij}$ of an edge between any $v_i \in
V_i$ and $v_j \in V_j$. For $2\le k\le 30$ we
determine the phase space: the combinations of edge and $k$-star densities achievable
asymptotically. For these models there are special points on the
boundary of the phase space with nonunique asymptotic (graphon)
structure; for the 2-star model we prove 
that the nonuniqueness extends to 
entropy maximizers in the interior of the phase space.

\end{abstract}



\section{Introduction}
\label{SEC:Intro}


We study the asymptotics of large, simple, labeled graphs
constrained to have certain fixed \emph{densities} $t_j$
of subgraphs $H_j$, $1\le j\le \ell$ {(see definition below)}.
To study the asymptotics we use the graphon formalism of Lov\'asz
{\it et al} \cite{LS1, LS2,BCLSV, BCL, LS3}
and the large deviations theorem of
Chatterjee and Varadhan \cite{CV}, from which one can reduce the analysis to the
study of the graphons which maximize the entropy subject to the
density constraints,  as in the previous works \cite{RS1, RS2, RRS}.

Most of this work considers the simple cases, called \emph{$k$-star
  models}, in which $\ell=2$, $H_1$ is an edge, and $H_2$
is a ``$k$-star'': $k\ge 2$ edges with a common vertex.  For these models we prove that all graphons which
maximize the entropy, subject to any realizable values of the density constraints,
are ``multipodal'': there is a partition of the
vertices into $M < \infty$ subsets $V_1, V_2, \ldots, V_M$, and a set
of well-defined probabilities $g_{ij}$ of an edge between any $v_i \in
V_i$ and $v_j \in V_j$. In particular the optimizing graphons are
piecewise constant, attaining only finitely many values. 

For any finite set of constraining subgraphs $H_j$, $1\le j\le \ell$, one can consider the
\emph{phase space} (also called the \emph{feasible region}), 
the subset of the unit cube in $\bbR^\ell$ consisting of accumulation
points of all densities $t=(t_1,\ldots t_\ell)$ achievable by finite
graphs.  The phase space for the 2-star model, and all
graphons corresponding to (that is, with densities on) boundary points of the phase space, were derived in
\cite{AK}.  {We derive} the phase space and bounding graphons
for $k$-star models with $2\le k \le 30$.
In these models there are distinguished points on the boundary
for which 
the graphon is not unique. For the 2-star model we prove that this
nonuniqueness extends to the entropy maximizer in the {interior} of the
phase space.

Extremal graph theory, the study of the boundaries of 
the phase spaces of networks, has a
long and distinguished history; see for instance \cite{B}. Few examples have been
solved, the main ones corresponding to two constraints: edges and 
the complete graph $K_p$ on $p\ge 3$ vertices \cite{R}, which includes
the edge/triangle model discussed below. {In these examples the optimal
graphs are multipodal. However more recently \cite{LS3} examples of
`finitely forced graphons' were found: these are (typically non-multipodal)
graphons uniquely determined by
the values of (finitely many) subgraph densities}. Although we are
mainly interested in entropy-maximizing graphons for constraints in
the \emph{interior} of the phase space, where one can define
phases, (see \cite{RS1} and references therein),
these results for graphons corresponding to boundary points are clearly
relevant to our study, and will be discussed in Section \ref{finitely_forced} and the Conclusion.

The significance of multipodal entropy optimizers emerged in a series 
of three
papers \cite{RS1,RS2,RRS} on a model with constraints on edges
and triangles, rather than edges and stars. In the edge/triangle model
evidence, but not proof, was given that entropy optimizers were
$M$-podal throughout the whole of the phase space, $M$
growing without bound as the two densities approach 1. Here we {\it
  prove} that all optimizers are $M$-podal, with a uniform bound on $M$, 
in any $k$-star model.

A related but different family of models consists of 
exponential random graph models (ERGMs): see for instance \cite{N, Lov} and the many references
therein. In physics terminology the models in \cite{RS1,RS2,RRS} and
this paper are microcanonical whereas the ERGMs based on the same
subgraph densities are the corresponding grand canonical versions
or ensembles. In distinction with statistical mechanics with short
range forces \cite{Ru,TET}, here the microcanonical and grand
canonical ensembles are inequivalent \cite{RS1}:
because the relevant entropy function
is not concave on the phase space, there are large portions of the phase space where the ERGM model gives no information about
the constrained optimal graphons.
On the other hand, all information about the grand canonical
ensemble can be derived from the microcanonical ensemble.
These facts have important
implications regarding the notion of \emph{phases} in random graph models.

In the Conclusion
below we discuss this extent of the loss of information in ERGMs as
compared with microcanonical models.  Continuing the analogy with
statistical mechanics we also discuss the relevance of multipodal
structure in the study of emergent phases in all such parametric families of
large graphs, as the vertex number grows.

\section{Notation and background}
\label{SEC:notation}

Fix distinct positive integers $k_1,\dots, k_\ell,\  \ell \ge 2$, and consider simple (undirected, with
no multiple edges or loops) graphs $G$ with vertex set $V(G)$ of labeled vertices.
For each $k=k_i$, 
set $T_k(G)$ to be the set of graph homomorphisms from a $k$-star
into $G$. We assume $k_1=1$ so the $k_1$-star is an edge.
Let $n=|V(G)|$. The \emph{density} of a subgraph $H$ refers to the
relative fraction of maps from $V(H)$ into $V(G)$ which preserve edges:
the $k$-star density is 
\begin{equation}
t_{k}(G)\equiv \frac{|T_k(G)|}{n^{k+1}}.
\end{equation}

For $\alpha > 0$ and $\T=(\T_1,\dotsm\T_\ell)$ define
$\displaystyle Z^{n,\alpha}_{\T}$ to be the number of graphs on $n$ vertices with
densities
\begin{equation}
t_{k_i}(G) \in (\T_i-\alpha,\T_i+\alpha),\ 1\le i\le \ell.
\end{equation}
We sometimes denote $\T_1$ by $\E$ and $T_{1}(G)$
by $E(G)$. 

Define the \emph{(constrained) entropy density} $s_{\T}$ to be the exponential rate of growth of 
$Z^{n,\alpha}_{\T}$ as a function of $n$: 
\begin{equation}
s_{\T}=\lim_{\alpha\downarrow 0}\lim_{n\to \infty}\frac{\ln(Z^{n,\alpha}_{\T})}{ n^2}.
\end{equation}
The double limit defining the entropy density $s_{\T}$ is known to
exist \cite{RS1}. To analyze it we make
use of a variational characterization of $s_{\T}$, and 
for this we need further notation to analyze limits of graphs as
$n\to \infty$. (This work was recently
developed in \cite{LS1,LS2,BCLSV,BCL,LS3}; see also the recent book \cite{Lov}.)
The (symmetric) adjacency matrices of graphs on $n$ vertices
are
replaced, in this formalism, by symmetric, measurable functions $g:[0,1]^2\to[0,1]$; 
the former are recovered by using a partition of
$[0,1]$ into $n$ consecutive subintervals. The functions $g$ are
called graphons.

For a graphon $g$ define the \emph{degree function} $d(x)$ to be $d(x)=\int^1_0 g(x,y)dy$.
The $k$-star density of $g$, $t_k(g)$, then takes the simple form 
\begin{equation}
 t_k(g) =  \int_0^1 d(x)^k\,dx.
\end{equation}
Finally, the \emph{Shannon entropy density} (\emph{entropy} for short) 
of $g$ is
\begin{equation}
\S(g) = \frac12\int_{[0,1]^2} S[g(x,y)] \,dxdy,
\end{equation} 
where $S$ is the Shannon entropy function 
\begin{equation}\label{EQ:Shannon}
S(w) =  -w\log w -(1-w)\log(1-w).
\end{equation}

The following is a minor variant of a result in \cite{RS1}
(itself an adaption of a proof in \cite{CV}):
\begin{theorem}[The Variational Principle.]
For any feasible set $\T$ of values of the densities $t(g)$ we have 
$s_{\T} = \max [\S(g)]$, where the
maximum is over all graphons $g$ with $t(g)=\T$.
\end{theorem}

\noindent (Some authors use instead the \emph{rate function}
$I(g)\equiv-\S(g)$, and then minimize $I$.)
The existence of a maximizing 
graphon $g=g_{\T}$ for any constraint $t(g)=\T$ was proven in \cite{RS1},
again adapting a proof in \cite{CV}.
If the densities are that of one or more $k$-star subgraphs we refer
to this maximization problem as a \emph{star model}, though we
emphasize that the result applies much more generally \cite{RS1}.

We consider two graphs {\it equivalent} if they are obtained from
one another by relabeling the vertices. For graphons, the analogous operation is applying a
measure-preserving map $\psi$ of $[0,1]$ into itself, replacing $g(x,y)$ with $g(\psi(x),\psi(y))$, see \cite{Lov}. 
The equivalence classes of graphons under relabeling are called \emph{reduced graphons}, and on this space
there is a natural metric, the \emph{cut metric}, with respect to which graphons are equivalent if and only
if they have the same subgraph densities for all possible finite
subgraphs \cite{Lov}. In the remaining sections of the paper, 
to simplify the presentation we will sometimes
use a convention under which a graphon may be said to have some property
if there is a relabeling of the vertices such that it has the property.

The graphons which maximize the constrained entropy tell us what `most' or
`typical' large constrained graphs are like: if $g_{\T}$ is the only reduced graphon maximizing
$\S(g)$ with $t(g)=\T$, then as the number $n$ of vertices diverges and $\alpha_n\to
0$, exponentially most graphs with densities
$t_i(G)\in (\T_i-\alpha_n,\T_i+\alpha_n)$
will have reduced graphon close to $g_{\T}$ \cite{RS1}.
This is based on large deviations from \cite{CV}. 

\section{Multipodal Structure}
\label{SEC:multi}

\subsection{Monotonicity}
In this paper our graphons $g$ have a 
constraint on edge density, which is the integral
of $g$ over $[0,1]^2$, so we treat $g$ as a way of assigning this
conserved
quantity, which for intuitive purposes we term `mass', to the various
regions of $[0,1]^2$.


Except where otherwise indicated, in this section we restrict attention to a
$k$-star model for fixed $k\ge 2$.
Note that the values of $\T(g)$ are determined by the degree function 
$d(x)$, as $t_k = \int_0^1 d(x)^k dx$. 
We first prove a general result for graphons with constraints only on
their degree function. (See Theorem 1.1 in \cite{CDS} for a 
stronger result but with stronger hypotheses.)

\begin{theorem}\label{monotonic} If the degree function $d(x)$ is monotonic 
nondecreasing
and if the graphon $g(x,y)$ maximizes the entropy among graphons with the 
same degree function, then $g(x,y)$ is {doubly monotonic nondecreasing, that is,
outside a set of measure zero,
$g$ is monotonic nondecreasing
in each variable}.
\end{theorem}

Before providing a rigorous proof, consider the following heuristic. 
Suppose 
that $d(x)$ is monotonically nondecreasing and that $g(x_1,y_1) >g(x_2,y_1)$ 
for some $x_1<x_2$. Since $d(x_1)\le d(x_2)$, 
there must be some other value of $y$,
say $y_2$, such that $g(x_1,y_2) < g(x_2,y_2)$. But then moving mass from 
$(x_1,y_1)$ to $(x_2,y_1)$ and moving the same amount of mass from 
$(x_2,y_2)$ to $(x_1,y_2)$
(and likewise moving mass from $(y_1,x_1)$ to $(y_1,x_2)$ and from 
$(y_2,x_2)$ to $(y_2,x_1)$ to preserve symmetry)
will increase the entropy $\S$ while leaving the degree function $d(x)$ fixed.
The problem with this heuristic is that mass is distributed continuously, 
so we cannot speak of mass ``at a point''. Instead, we must smear out the
changes over sets of positive measure. The following proof is essentially
the above argument, averaged over all possible data $(x_1, x_2, y_1, y_2)$.

\begin{proof} 
For $0<a,b<1$, define
\be \eta(a,b)  =  \int_0^1 \max(g(a,y)-g(b,y),0)\,dy
\ee
and 
$\kappa(a,b)  = \eta(b,a).$
Note that if $b\ge a$ then $\kappa(a,b)-\eta(a,b) = d(b)-d(a) \ge 0$,
and that $g(x,y)$ being doubly monotonic almost everywhere is 
equivalent to $\eta(a,b)$ being almost everywhere zero.
For $b\ge a$ let 
\begin{equation}
\gamma(y,a,b) = \begin{cases} \eta(a,b), & g(b,y) \ge g(a,y) \cr 
\kappa(a,b), & g(b,y) < g(a,y) \end{cases},
\end{equation}
and for $a > b$ let $\gamma(y,a,b) = \gamma(y,b,a)$.
If $g(x,y)$ is doubly monotonic, then $g(b,y) \ge g(a,y)$ for $b \ge a$,
so $\gamma$ is equal to $\eta$ almost everywhere, which is equal to
zero almost everywhere. 

Now evolve the graphon $g$ according to the following integro-differential equation:
\begin{equation} \label{monotonicity-DE} \frac{d\phantom{t}}{dt}g_t(x,y) = \int_0^1  \gamma_t(y,x,b)[g_t(b,y)-g_t(x,y)]\,db
+ \int_0^1  \gamma_t(x,y,b) [g_t(x,b)-g_t(x,y)]\,db.
\end{equation}
Existence and uniqueness of solutions to this equation is straightforward. 
Working in the $L^\infty$ norm, 
the Picard iteration converges to a classical solution, i.e., a family $g_t$ of measurable functions $[0,1]^2 \to [0,1]$
that are pointwise differentiable with respect to $t$.
If $g(x,y)$ is doubly monotonic almost everywhere, 
then equation (\ref{monotonicity-DE}) simplifies to ${d g_t}/{dt}  = 0$. 
Otherwise, we will show that the flow preserves the degree function $d(x)$ and 
increases the entropy $\S(g)$, contradicting the assumption that $g$ is 
an entropy maximizer. 

To see that the degree function is unchanged we compute
\begin{eqnarray}
\frac{d\phantom{t}}{dt} d_t(x) & = & \int_0^1  \frac{d\phantom{t}}{dt}g_t(x,y)
\,dy \cr 
& = & \int_0^1\int_0^1  \gamma_t(y,x,b)[g_t(b,y)-g_t(x,y)] \,dy\,db \cr
&& + \int_0^1\int_0^1 \gamma_t(x,y,b)[g_t(x,b)-g_t(x,y)]\, dy\, db
\end{eqnarray}
The second integrand is anti-symmetric in $b$ and $y$, and so integrates to
zero. For the first integral, fix a value of $b$ and write
\begin{equation}
g_t(b,y)-g_t(x,y) = \max(g_t(b,y)-g_t(x,y),0) - \max(g_t(x,y)-g_t(b,y),0).
\end{equation}

{
Suppose for the moment that $b>x$.  
Then $\gamma(y,x,b) = \eta(x,b)$ if $g(b,y) > g(x,y)$, so the
integral of
$\gamma(y,x,b) (g(b,y)-g(x,y) dy$
over the values of $y$ where $g(b,y)>g(x,y)$ is the same as the integral of
$\eta(x,b) (g(b,y)-g(x,y)) dy$ over those same values,
namely
$\eta(x,b) \kappa(x,b).$
Since $\gamma(y,x,b) = \kappa(x,b)$ if $g(b,y) < g(x,y)$, the integral of
$\gamma(y,x,b) (g(b,y)-g(x,y) dy$
over the values of $y$ where $g(b,y)<g(x,y)$ is the same as the integral of
$\kappa(x,b) (g(b,y)-g(x,y)) dy$ over those values, namely
$-\kappa(x,b) \eta(x,b)$. Adding these together, the integral from $0$ to $1$ is
zero.

On the other hand, if $x>b$, then $\gamma(y,x,b)=\gamma(y,b,x)$, which is
$\eta(b,x)$ if $g(x,y)>g(b,y)$, and is $\kappa(b,x)$ if $g(x,y)<g(b,y)$. Once
again we divide the region of integration for $y$ into two zones depending
on which inequality applies, and contributions of the two regions cancel. Since
the inner integral (over $y$) gives zero for all values of $b$, the double
integral over $b$ and $y$ is zero.
}

To see that $\S(g_t)$ is increasing, we compute
\begin{eqnarray} \frac{d\phantom{t}}{dt}\S(g_t) & = & \iint S'(g_t(x,y)) \frac{d\phantom{t}}{dt}g_t(x,y)\, dx\,dy \cr 
& = & \iint S'(g_t(x,y))\, dx\, dy \int_0^1  \gamma_t(y,x,b)[g_t(b,y)-g_t(x,y)] \cr
&  & \hskip1truein + \gamma_t(x,y,b)[g_t(x,b)-g_t(x,y)] \,db\cr 
& = & 2 \iint S'(g_t(x,y))\, dx\, dy \int \gamma_t(x,y,b)[g_t(x,b)-g_t(x,y)] \,db\cr 
& = & \iiint \gamma_t(x,y,b)[g_t(x,b)-g_t(x,y)][S'(g_t(x,y))-S'(g_t(x,b))]\, dx\,dy\,db
\end{eqnarray}
However, $[g_t(x,b)-g_t(x,y)][S'(g_t(x,y))-S'(g_t(x,b))]$ is strictly positive
when $g_t(x,b) \ne g_t(x,y)$, thanks to the {concavity} of the function $S$. 
So ${d \S(g_t)}/{dt}$ is non-negative, and is strictly positive unless $\gamma_t(x,y,b)$
is zero for (almost) all triples $(x,y,b)$ for which $g_t(x,b) \ne g_t(x,y)$. 
Since the vanishing of $\gamma_t$ is equivalent to the double monotonicity of 
$g_t$, a maximizing $g$ must be doubly monotonic. 
\end{proof}

If $g(x,y)$ is doubly monotonic almost everywhere, we can adjust it on 
a set of measure zero to be doubly monotonic everywhere. Just take 
the adjusted value $\tilde g(a,b)$ 
to be the essential supremum of $g(x,y)$ over all $(x,y)$
with $x<a$ and $y<b$. Since changes over sets of measure zero have no effect on the integrals of $g$, we can therefore assume hereafter that $g(x,y)$ is doubly monotonic whenever $g$ is an entropy maximizer.

We previously defined multipodality in terms of graphs. Here we rephrase the
definition directly in terms of graphons. A graphon $g$ is $M$-podal if the interval $[0,1]$ can be 
split into $M$ regions, called ``clusters'',  such that the value of $g(x,y)$ depends only on which cluster $x$ is in and 
which cluster $y$ is in. 
After applying a measure-preserving transformation of $[0,1]$, we can assume that the clusters are consecutive intervals
$I_1, \ldots, I_M$. The graphon $g$, viewed as a function of two variables, then resembles 
a checkerboard, being constant on each rectangle $I_i \times I_j$.

As a first corollary to this theorem, we have the following result
generalizing slightly a result of \cite{CDS}: 

\begin{proposition}\label{CDSprop}
If $g$ maximizes the entropy among graphons with the 
same degree function, and if the degree function $d$ takes only $M$ values, then $g$ is $M$-podal.
\end{proposition}

\begin{proof} We first apply a measure-preserving bijection of $[0,1]$ 
to make $d(x)$ monotone increasing. If $d$ is
  constant on some interval $[x_1,x_2]$ then we claim that for any $y$, $g(x,y)$ is
  also constant for $x\in[x_1,x_2]$, since if $g(x_2,y)>g(x_1,y)$ then there
  would have to be some $y'$ such that $g(x_2,y')<g(x_1,y')$ to assure
  that $d(x_1)=d(x_2)$, contradicting the monotonicity of $g$.
\end{proof}

As a second corollary, we obtain a strong continuity result:
\begin{proposition}If $g$ maximizes the entropy among graphons with a 
given degree function, then $g$ is continuous almost everywhere. 
\end{proposition}

\begin{proof}
We may assume that $g$ is doubly monotonic. This implies that $g$ is monotonic
(and of course bounded) on any line $y=x+c$ of slope 1. $dg$ is then a bounded
measure on this line, whose discrete part is supported on a countable number of 
points. In other words, $g(x,x+c)$ can only have a countable number of jump
discontinuities as a function of $x$, and is otherwise continuous. By Fubini's
theorem, $g(x,y)$ must then be continuous in the $(1,1)$ direction for
almost all $(x,y)$. But if 
$g$ is continuous in the $(1,1)$ direction at a point $(a,b)$, then for each
$\epsilon>0$ we can find a $\delta>0$ such that
$g(a+\delta,b+\delta)$ and $g(a-\delta,b-\delta)$ are both within $\epsilon$
of $g(a,b)$. But since $g$ is doubly monotonic, $g(a-\delta,b-\delta)\le
g(x,y) \le g(a+\delta,b+\delta)$ for all $x \in (a-\delta, a+\delta)$ and 
all $y \in (b-\delta, b+\delta)$, so $g$ is continuous at $(a,b)$. 
\end{proof}

The upshot of this proposition is that the value of $g$ at a generic point
$(a,b)$, and the functional derivatives $\delta \S/\delta g$ and
$\delta t_k/\delta g$
at $(a,b)$, control the values of these functions in
a neighborhood of $(a,b)$.
(By the functional derivative $\delta t_k/\delta g$ 
we mean the function such that 
$\int \delta t_k/\delta g(x,y)\ \delta g(x,y)\, dxdy$ 
is the
linear term in the expansion of $t_k(g+\delta g)$.)
We can therefore do functional calculus computations
at points $(a,b)$ and $(c,d)$, and then speak of moving mass from a 
neighborhood of $(a,b)$ to a neighborhood of $(c,d)$. In other words, once
we have (almost everywhere) continuity, {informal arguments such as those 
preceding} the proof of Theorem \ref{monotonic} can be
used directly. 

To be more precise, let $\rho_1$ and $\rho_2$ be symmetric bump functions on $[0,1]^2$ each of total integral 1, 
with $\rho_1$ supported on small neighborhoods of 
$(a,b)$ and $(b,a)$, and with $\rho_2$ supported on small neighborhoods of $(c,d)$ and $(d,c)$. When we speak of moving mass $\epsilon$ from $(a,b)$ to $(c,d)$, we
mean changing $g(x,y)$ to $g'(x,y) = g(x,y) + \epsilon \rho_2(x,y) - \epsilon \rho_1(x,y)$ at each point $(x,y)$. 
As long as $g(x,y)$ is continuous at $(a,b)$ and $(c,d)$ and is neither 0 nor 1 there, this brings about the following change to the entropy:
\begin{equation}
\S(g')-\S(g) = \epsilon \iint \frac{\delta \S}{\delta g(x,y)} \big (\rho_2(x,y)-\rho_1(x,y)\big ) dx\, dy +O(\epsilon^2) \approx 
\epsilon \left [\frac{\delta \S}{\delta g(c,d)}
- \frac{\delta \S}{\delta g(a,b)}\right ].
\end{equation}
If $\frac{\delta \S}{\delta g(c,d)} > \frac{\delta \S}{\delta g(a,b)}$, then we can always 
pick the supports of $\rho_1$ and $\rho_2$ small enough, and the value of $\epsilon$ small enough, that the resulting 
change in entropy is positive.  Similar considerations apply to the $k$-star density $t_k$.

\subsection{$M$-podality}
Here is our main theorem. 

\begin{theorem} \label{maintheorem} For the $k$-star model, any
  graphon $g$ which maximizes the entropy $\S(g)$ and is constrained
  by $t(g)=\T$, is $M$-podal for some $M<\infty$.
\end{theorem}

\begin{proof} 
When $\T$ is on the boundary of the phase space (the space of
achievable values of the densities $t(g)$)
Theorem \ref{clique_theorem}, below, indicates that the 
graphon is either 1-podal (on the lower boundary) or $\le 3$ podal
on the top boundary. 
So for the remainder of the proof we assume $\T$ is in the interior.

\begin{lemma}\label{lagrange1} For the $k$-star model, let $g$ be a graphon 
  that maximizes the entropy $\S(g)$ subject to the  constraints
 $t(g)=\T$, where $\T$ lies in the interior of the phase space of possible
 densities. Then there exist constants $\beta_1$, $\beta_2$ such that
 the Euler-Lagrange equation
 \begin{equation}\label{EL1}
\beta_1 + \beta_2d(x)^{k-1}+\beta_2 d(y)^{k-1}=\ln \left [ \frac{1}{g(x,y)} - 1 \right ]
\end{equation}
holds for almost every $(x,y)$. Furthermore, the constants $\beta_{1},\beta_2$ are uniquely defined. 
\end{lemma}

\begin{proof}[Proof of lemma]
First note that $g$ cannot take values in $\{0,1\}$ only, since such a
graphon would have zero entropy, and for each $\T$ in the interior of
the phase space it is easy to construct a graphon (even bipodal, see
section \ref{SEC:simul}) with positive entropy.

%
%
%

We claim that $g\in(0,1)$ on a set of full measure. Suppose otherwise.
Then (by double monotonicity) $g(x,y)$ is either 1 on a neighborhood of 
$(1,1)$ or is zero on a neighborhood of $(0,0)$ (or both). By moving 
$\epsilon$ mass from a neighborhood of $(1,1)$ to a neighborhood of 
$(0,0)$, we can increase the entropy by order $\epsilon \ln(1/\epsilon)$,
while leaving the edge density fixed. In the process, we will decrease 
$t_k$, since the monotonicity of $d(x)$ implies that the
functional derivative $\delta t_k/\delta g =
(k/2) (d(x)^{k-1}+d(y)^{k-1})$ is greater near $(1,1)$ than near $(0,0)$.
However, we claim that we can restore the value of $t_k$ by moving mass 
within the region $R_0$ where $g \in (0,1)$, at a cost in entropy
of order $\epsilon$. Since $\epsilon \ln(1/\epsilon) \gg \epsilon$, 
for sufficiently small $\epsilon$ the
combined move will increase entropy while leaving $t_1$ and $t_k$ fixed,
which is a contradiction.

The details of the second movement of mass depend on whether $\delta t_k/
\delta g$ is constant on $R_0$ or not. If $\delta t_k/\delta g$ is not constant,
we can restrict attention to a slightly smaller region $\tilde R_0$ 
where $g$ is bounded
away from 0 and 1, and then move mass from a portion of $\tilde R_0$ 
where $\delta t_k/\delta g$ is smaller
to a portion where $\delta t_k/\delta g$ is larger. This will increase 
$t_k$ to first order in the amount of mass moved. Since $g$ is bounded away
from $0$ and $1$ on $\tilde R_0$, the change in the entropy is bounded by 
a constant times the amount of mass moved, as required.

If $\delta t_k/\delta g$ is constant on $R_0$, then it is constant on
a rectangle within $R_0$. But the only way for $d(x)^{k-1}+d(y)^{k-1}$
to be constant on a rectangle is for $d(x)$ and $d(y)$ to be constant
for $(x,y)$ in that rectangle. By Theorem \ref{monotonic}, this implies that
$g(x,y)$ is constant on the rectangle. Moving mass within the
rectangle will then change neither the entropy nor $t_k$ to first order,
but will change both (with $t_k$ increasing and the entropy
decreasing) to second order. So by moving an amount of mass of order
$\sqrt{\epsilon}$, we can restore the value of $t_k$ at an
$O(\epsilon)$ cost in entropy. This proves our assertion that $g(x,y)\in
(0,1)$ on a set of full measure.

Next we note that the degree function $d(x)$ must take on at least two values, since otherwise
we would be on the lower boundary of the phase space, with $t_k=t_1^k$. This means that the functional derivative
\begin{equation}\label{func_deriv}
 \frac{\delta t_k}{\delta g}(x,y) = \frac{k}{2} \left ( d(x)^{k-1} +
 d(y)^{k-1} \right )
\end{equation}
is not a constant function.  If $\delta \S/\delta g(x,y)$ cannot be written as a linear combination
of $\delta t_1/\delta g(x,y)=1$ and $\delta t_k/\delta g(x,y)$, then we can find three points
$p_i=(x_i,y_i)$ such that $g$ is continuous at each point, and such that the matrix 
\begin{equation}
 \begin{pmatrix} 1 & 1&1 \cr 
\frac{\delta t_k}{\delta g}(p_1) 
 & \frac{\delta t_k}{\delta g}(p_2) 
 & \frac{\delta t_k}{\delta g}(p_3) \cr  
\frac{\delta\S}{\delta g}(p_1)  
& \frac{\delta\S}{\delta g}(p_2)  
& \frac{\delta\S}{\delta g}(p_3)
\end{pmatrix}
\end{equation}
is invertible.  But then, by adjusting the amount of mass near each $p_i$ (and near the reflected points
$p_i'=(y_i,x_i)$), we can independently vary $t_1$, $t_k$ and $\S$ to first order. By the inverse function theorem, 
we can then increase $\S$ while leaving $t_1$ and $t_k$ fixed, which is a contradiction.  Thus $\delta \S/\delta g$ must
be a linear combination of $\delta t_1/\delta g$ and $\delta t_k/\delta g$, which is equation (\ref{EL1}). 
Since $\delta t_1/\delta g$ and $\delta t_k/\delta g$ are linearly independent, the coefficients are unique. 
\end{proof}

Continuing the proof of Theorem \ref{maintheorem}, solving (\ref{EL1}) for $g(x,y)$ gives
\begin{equation}\label{Euler-Lagrange}
g(x,y)= \frac{1}{1+ \exp[{\beta_1 + \beta_2d(x)^{k-1}+\beta_2 d(y)^{k-1}}]}, 
\end{equation}
and integrating with respect to $y$ gives
\begin{equation}\label{degree eqn}
d(x)=\int^1_0 {dy\over  1+ \exp[{\beta_1 + \beta_2d(x)^{k-1}+\beta_2 d(y)^{k-1}}]}. 
\end{equation}
Let $d(x)$ be any solution of (\ref{degree eqn}), let $z$ be a real
variable, and consider the function
\begin{equation}\label{defineF}
F(z)=z -\int^1_0 {dy\over  1+ \exp[{\beta_1 + \beta_2z^{k-1}+\beta_2d(y)^{k-1}}]},
\end{equation}
where the function $d(y)$ is treated as given. 
By equation (\ref{degree eqn}), all actual values of $d(x)$ are roots 
of $F(z)$.  

The second term in (\ref{defineF}) 
is an analytic function of $z$, as follows. 

Write $W=\beta_2 z^{k-1}$ and $Y=\beta_2 d(y)^{k-1}$ then the integral
is 
\begin{equation}
\int {d\mu(Y)\over  1+ \exp({\beta_1 + W +Y})},
\end{equation}
the convolution of an analytic function of $W$ with an integrable
measure $\mu(Y)$. Since the Fourier transform of an analytic function decays exponentially at infinity
and the Fourier transform of an integrable measure is bounded, the Fourier transform of the
convolution decays exponentially at infinity, so the convolution itself  is an analytic function of $W$. Since $W$
is an analytic function of $z$, $F(z)$ is an analytic function of $z$. 

Note that $F(z)$ is strictly negative for $z \le 0$ and strictly positive
for $z \ge 1$ (the integrand being strictly less than $1$).  
Being analytic and not
identically zero, $F(z)$ can only have finitely many roots in any
compact interval.  
By Rolle's Theorem any accumulation point
   of the roots would have to be an accumulation point of the roots
  of $F'(z)$, $F''(z)$, etc. So all derivatives of $F$ would have to
  vanish at that point, making the Taylor series around it identically
  zero.
In particular $F(z)$ can only have finitely many roots between 0 and 1, 
implying there are
only finitely many values of $d(x)$.  By Proposition \ref{CDSprop}, the graphon $g$
is $M$-podal.
Note that the  roots of $F(z)$ are not necessarily values of $d(x)$,
so this construction only gives an upper bound to the actual
value of $M$. 
\end{proof}

The previous proof actually showed more than that optimal graphons are
multipodal. We showed that the possible values
of $d(x)$ are roots of the function $F(z)$ defined in
(\ref{defineF}).  This allows us to prove the following refinement of 
Theorem \ref{maintheorem}

\begin{theorem}\label{maintheorem2}
For any $k$-star model, there exists a fixed $M$ such that all
entropy-maximizing graphons are $m$-podal with $m \le M$.
\end{theorem}

\begin{proof}
If $f(z)$ is an analytic function on a compact interval with $m$ roots
(counted with multiplicity), then any $C^m$-small perturbation of
$f(z)$ will also have at most $m$ roots on the interval. 
Thus if $g$ is a graphon with associated function $F(z)$ with $m$ roots, 
and if $\tilde g$ is an optimal graphon whose degree function is $L^1$-close 
to that
of $g$, then from the convolution, the associated function $\tilde F(z)$ of $\tilde g$ will 
be a $C^m$-small perturbation of $F(z)$ and so will have at most
$m$ roots, and $\tilde g$ will be at most $m$-podal. 

Suppose there is no universal bound $M$ on the podality
of optimal graphons on $R$. Let $g_1, g_2, \ldots$ be a sequence of
optimal graphons (perhaps with different values of $\T$) with
the podality going to infinity, and let $F_i(z)$ be the associated
functions for these graphons. Since the space of graphons is
compact, there is a subsequence that converges to a graphon
$g_\infty$.  The associated
function $F_\infty(z)$ of $g_\infty$ is analytic, and so has only a finite number
$M_\infty$ of roots. But then
for large $i$, $F_i(z)$ has at most $M_\infty$ roots and $g_i$ is at most
$M_\infty$-podal, which is a contradiction. 
\end{proof}

We end this section with an argument which displays the use of the
notion of phase. By definition a phase is a connected open subset of
the phase space in which the entropy $s_\T$ is analytic. (The
connection with statistical mechanics is discussed in the Conclusion.)
The following only simplifies one step in the proof of our main result,
Theorem \ref{maintheorem}, but it shows how the notion of phase can be
relevant. 

\begin{theorem}\label{zero}
For any microcanonical model let $g_0$ be a graphon 
which maximizes the Shannon entropy $\S(g)$ subject to the  constraints
$t(g)=\T$, where $\T$ lies in the interior of the phase space of possible
densities and $s_{\T}$ is differentiable at $\T$.
Then the set $A=g_0^{-1}(\{0,1\})$ has measure zero.
\end{theorem}

\begin{proof}[Proof of theorem]

Define $g_\epsilon$ by moving the value of
$g_0$ on $A$ by $\epsilon$, $0<\epsilon<1$. 

From their definitions the densities satisfy
$t(g_\epsilon)=t(g_0)+O(\epsilon)$. By integrating over $A$ we see that the
Shannon entropy satisfies 

\begin{equation}
\S(g_\epsilon)=\S(g_0) - m(A)\epsilon \ln(\epsilon), 
\end{equation}
so noting that $s_{t(g_\epsilon)}\ge \S(g_\epsilon)$ and $s_{\tau_0}=\S(g_0)$
we get 

\begin{equation}
s_{t(g_\epsilon)}\ge s_{\tau_0}- m(A)\epsilon
\ln(\epsilon). 
\end{equation}
From differentiability, as the vector $\alpha \to 0$
\begin{equation}\label{differentiable}
  s_{\tau_0 + \alpha}=s_{\tau_0}+O(||\alpha||). 
\end{equation} 
So as $\epsilon \to 0$ we have a 
contradiction with (\ref{differentiable})
unless $m(A)=0$, which concludes the proof.

\end{proof}


\section{Phase space}
\label{SEC: space}

We now consider $k$-star models with $2\le k \le 30$.
The phase space is the set of those $(\E,\T)\subset [0,1]^2$ 
which are accumulation points of the values of pairs (edge density,
$k$-star density) for finite graphs. The lower boundary (minimum of $\T$ given
$\E$) is easily seen to
be the Erd\H{o}s-R\'enyi curve: $\T=\E^k$, since H\"older's inequality
gives
\begin{equation}
\T^{1/k} = \|d(x)\|_k \ge \|d(x)\|_1 = \E.
\end{equation} 

We now determine the upper
curve. This was determined for $k=2$ in \cite{AK}, and perhaps was
published for
higher $k$ though we do not know a reference. 

We are looking for the graphon which maximizes $t(g)$ for fixed
$e(g)=\E$, and this time arrange the points of the line so that $d(x)$ is 
monotonically decreasing. As in the proof of Theorem \ref{maintheorem} we can assume that $g$ is monotonic (this time, 
decreasing) in both coordinates.

We call a graphon a {\em g-clique} if it is bipodal of the form
\begin{equation}\label{EQ:Clique}
 g(x,y) = \begin{cases} 1 & x<c \text{ and }y < c \cr 0
 & \hbox{otherwise} \end{cases}
\end{equation}
and a {\em g-anticlique} if it is of the form 
\begin{equation}\label{EQ:Anticlique}
 g(x,y) = \begin{cases} 0 & x>c \text{ and }y > c \cr 1
 & \hbox{otherwise.} \end{cases}
\end{equation}

\begin{theorem} \label{clique_theorem}
For fixed $e(g)=\E$,  and any $2 \le k \le 30$, any
graphon that maximizes the $k$-star density is equivalent to a g-clique or
g-anticlique.
\end{theorem}

G-cliques always have $c = \sqrt{\E}$ and $k$-star density
$\E^{(k+1)/2}$.  G-anticliques have $c=1-\sqrt{1-\E}$
and $k$-star density 
\begin{equation}
 c+c^k-c^{k+1}= [1-\sqrt{1-\E}][1+\sqrt{1-\E}(1-\sqrt{1-\E})^{k-1}].
\end{equation}
For $\E$ small, the g-anticlique has $k$-star density ${\E}/{2} +
O(\E^2)$, which is greater than $\E^{(k+1)/2}$.
For $\E$ close to 1, however, the g-clique has a higher $k$-star
density
than the g-anticlique.  

While our proof only covers $k$ up to 30, we conjecture that the result 
holds for all values of $k$. 
The only difficulty in extending to all $k$ is 
comparing the $t_k$-value for a clique, anticlique and a certain 
``tripodal anticlique'' of the form (\ref{c3=1}), below.

\begin{corollary}For 
$2 \le k \le 30$, the upper boundary of the phase space is
\begin{equation}\label{EQ:Upper Bdary}
\T = \begin{cases} [1-\sqrt{1-\E}][1+\sqrt{1-\E}(1-\sqrt{1-\E})^{k-1}]
  &
      \E \le \E_0 \cr
\E^{(k+1)/2} & \E \ge \E_0
\end{cases}
\end{equation}
where $\E_0$ is the value of $\E$ where the two branches of $\T=\T(\E)$ cross.
\end{corollary}
For $k=2$ the crossing point is
$\E_0=1/2$, for $k=3$ it is $\E_0=3/4$, and as $k \to \infty$ it
approaches $1$. The boundary of the phase space for the 2-star model
is shown in Fig.~\ref{FIG:Phase-Boundary}.
\begin{figure}[ht]
\centering
\includegraphics[angle=0,width=0.4\textwidth]{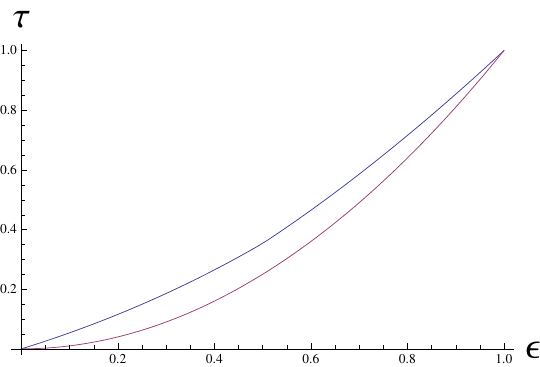}\hskip
1cm 
\includegraphics[angle=0,width=0.39\textwidth]{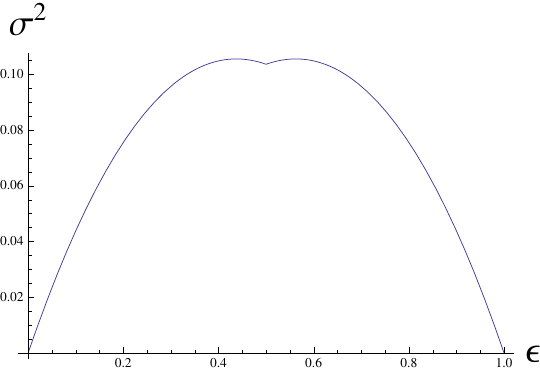} 
\caption{Boundary of the phase space for the 2-star case. Left: true
  phase boundary; Right: Plot of $\E$ versus $\sigma^2=\T-\E^2$; in
  this case the lower boundary becomes the $x$-axis.}
\label{FIG:Phase-Boundary}
\end{figure}

\subsection{Proof of Theorem \ref{clique_theorem}}

The proof has three steps:
\begin{enumerate}
\item Showing that $g(x,y)$ only takes on the values 0 and 1.
\item Showing that $g$ is at most $3$-podal.
\item Showing that $g$ is bipodal. (This is the only step that uses $k\le 30$.)
\end{enumerate}

\subsubsection{Step 1: Showing that we have a 0-1 graphon}
\label{SEC: 0-1}

The variational equation for maximizing $t_k(g)$ while fixing $t_1(g)$ is
(see (\ref{EL1}) without the `$\S$' term)
\begin{equation}\label{ddlambda}
d(x)^{k-1} + d(y)^{k-1} =  \lambda
\end{equation}
for some unknown constant $\lambda$, whenever 
$0< g(x,y) < 1$. 
When $g(x,y)=0$ we have $d(x)^{k-1}+d(y)^{k-1} \le \lambda$, and when $g(x,y)=1$
we have $d(x)^{k-1}+d(y)^{k-1}  \ge \lambda$.  
Since $ d(x)^{k-1}+d(y)^{k-1}$ is a decreasing function of both $x$ and $y$, this
means that there is a strip (of possibly zero thickness) running
roughly from the northwest corner of $[0,1]^2$ to the southeast
corner, where $g(x,y)$ is strictly between 0 and 1, and where
$ d(x)^{k-1}+d(y)^{k-1}  =\lambda$. All points to the southwest of this strip have
$g(x,y)=1$, and all points to the northeast have $g(x,y)=0$. (The
boundaries of the strip
are necessarily monotone paths).

We claim that the strip has zero thickness, and hence is actually a
boundary path between the $g=1$ zone and the $g=0$ zone. To see this,
suppose that the strip contains a small ball, and hence contains a
small rectangle. Since $d(x)^{k-1}+d(y)^{k-1}$ is constant on this rectangle,
$d(x)$ is constant and $d(y)$ is constant. We can then increase the
k-star density (i.e. $k$-th moment of $d(x)$) by moving some mass from
right to left, and from top to bottom on the mirror image region, as in the proof of Theorem \ref{maintheorem}. 
As in that proof, this move
increases $t_k(g)$ to second order in the amount of mass moved.
Since we assumed that our graphon was a maximizer (and not just a
stationary point), we have a contradiction.

\subsubsection{Step 2: Showing $g$ is at most $3$-podal}
\label{SEC:zig-zag}

By the above argument, the boundary between the region where $g=0$ and the region where $g=1$ is
a monotone decreasing path $\gamma$ from $(0,1)$ to $(1,0)$, symmetric under reflection in $\{x=y\}$. 
Let $(c,c)$ be the point where it crosses the line $\{x=y\}$, and let
$\gamma_0$ be the portion of $\gamma$ between $(0,1)$ and $(c,c)$.
We compute
\begin{eqnarray}  t_k  =&  \int_0^1 y^k dx & \hbox{since $d(x)=y(x)$}
\cr 
 =&   \int_{\gamma_0} y^k dx - x^k dy & \hbox{using reflection symmetry} \cr 
 =&  -c^{k+1} + \int_{\gamma_0} (y^k + k x^{k-1} y) dx & \hbox{since $-x^k dy = 
k x^{k-1} y dx - d(x^ky)$} \cr 
 =&  -c^{k+1} + \int_0^c (y^k + k x^{k-1} y) dx & \hbox{since there are no $dy$
terms.} 
\end{eqnarray}

Fix $c$. To maximize $t_k$ for this value of $c$ 
we need to maximize, over $y$, the above integral subject to the constraints that:
first, $y=y(x)\in[c,1]$ is a non-increasing 
function of $x$, 
and second, the integral $\int_0^c y dx$  is fixed. 
If there are any nontrivial functions $\delta y(x)$ such that $\int_0^c \delta y(x) dx = 0$ and such that 
$y(x) + t \delta y(x)$ is non-increasing (as a function of $x$) 
for all values of $t$ 
sufficiently close to 0, then 
$f(t) = \int_0^c (y +t \delta y)^k + k x^{k-1} (y + t \delta y) dx$ must have a local maximum at $t=0$. However, $f''(0) = k(k-1) \int_0^c y^{k-2} \delta y^2 dx > 0$. Thus such a function $\delta y$ cannot exist. 

{We now claim that $y$ takes at most one value between $c$ and $1$.
Suppose that the range of $y$ includes two points $y_1$ and $y_2$ that are strictly between $c$ and $1$.  That is, suppose the preimages of 
$(y_i-\epsilon, y_i+\epsilon)$ have positive measure for any $\epsilon > 0$.  
Let 
$\phi_1$ and $\phi_2$ be a positive bump functions supported in non-overlapping $\epsilon$-neighborhoods of $y_1$ and $y_2$.  
Then take
\begin{equation}
\delta y(x) =  C_1 \phi_1(y(x)) - C_2 \phi_2(y(x)),
\end{equation}
where the constants $C_1$ and $C_2$ are chosen so that 
$\int \delta y(x) dx = 0$. 
Let $M$ be the maximal value of $|C_i \phi_i'(y) |$, which is finite 
because $\phi_i$
is smooth. Then $y(x) + t \delta y(x)$ is monotonic for all $|t| < M^{-1}$, and thus $y$ is not maximal.
This completes the proof of the claim. 

Since $y$ takes at most one value between $c$ and $1$, $y(x)$ takes the form 
\be y(x) = \begin{cases} 1 & x < x_1 , \cr y_0 & x_1 < x < x_2, \cr  
c  & x > x_2. \end{cases}
\ee

There are now several cases to consider.
\begin{itemize}
\item
Suppose that $x_2<c$. Then we can vary $x_2$ and $y_0$ keeping $t_1$ and $c$ fixed to increase $t_k$:
we write
\begin{equation}
t_k=x_1^k(1-y_0)+y_0^k(x_2-x_1)+x_2^k(y_0-c)+c^k(c-x_2)
\end{equation}
and substituting $x:=x_2-x_1$ and $z:=y_0-c$ this becomes
\begin{equation}
t_k=-x_1^k z+(z+c)^kx+(x+x_1)^kz-c^k x+constant
\end{equation}
and $t_1=2xz+constant$, where $constant$ refers to positive quantities independent of $x_2,y_0$. 
Let $C=xz$ be fixed, so that $t_1$ is fixed. 
Then we have
\begin{equation}
\frac{t_k}C=\frac{(z+c)^k-c^k}z + \frac{(x+x_1)^k-x_1^k}x+ constant
\end{equation}
which, replacing $x$ with $C_2/z$, is a polynomial in $z$ and $1/z$ with nonnegative coefficients, hence convex. 
It is thus maximized at the endpoints of definition of $x,z$,
that is, $y_0=1$ or $x_2=c$. 

\item
Suppose that $x_2=c$ and $x_1>0$.  In this case we vary each of $x_1,c,y_0$
forwards and backwards in ``time'' according to the differential equation
\be \dot x_1(t)=\frac{\alpha}{1-y_1(t)},\qquad 
\dot c(t) = \frac{\beta}{y_0(t)-c(t)},\qquad 
\dot y_0(t) = \frac{\gamma}{c(t)-x_1(t)},
\ee
where a dot denotes a derivative with respect to the time parameter $t$, 
and with constants $\alpha,\beta,\gamma$ satisfying $\alpha+\beta+\gamma=0$.
Then $\dot t_1=0$.
We will show that the second 
derivative $\ddot t_k$ is positive at $t=0$ 
for some choice of $\alpha,\beta,\gamma$,
implying that we are not at a local maximum of $t_k$. 

We compute
\be \dot t_k = \alpha(kx_1^{k-1}+\frac{1-y_0^k}{1-y_0}) + 
\beta(kc^{k-1}+\frac{y_0^k-c^k}{y_0-c})+\gamma(ky_0^{k-1}
+\frac{c^k-x_1^k}{c-x_1}).\ee
Differentiating again, we have 
\be \ddot t_K = \alpha^2A+\beta^2 B+\gamma^2 C+\alpha\gamma D +\beta\gamma E \ee
where the coefficients
\begin{eqnarray} A & = & k(k-1) x_1^{k-2}/(1-y_0), \cr
B & = & (k(k-1) c^{k-2}
+ \frac{\partial}{\partial c} \frac{y_0^k-c^k}{y_0-c}   )/(y-c), \cr
 C &=& k(k-1) y_0^{k-2} /(c-x_1), \cr
D & = & \frac{1}{c-x_1}  \frac{\partial}{\partial y_0} \frac{1-y_0^k}{1-y_0}
+ \frac{1}{1-y_0}  \frac{\partial}{\partial x_1} \frac{c^k-x_1^k}{c-x_1}, \cr 
E & = & \frac{1}{c-x_1}  \frac{\partial}{\partial y_0} \frac{y_0^k-c^k}{y_0-c}
+ \frac{1}{y_0-c}  \frac{\partial}{\partial c} \frac{c^k-x_1^k}{c-x_1},
\end{eqnarray}

are all positive.
Now taking $\alpha=-E$ and $\beta=D$ (evaluated at $t=0$) 
for example, the last two terms cancel and we find $\ddot t_k(0)>0$. 

\item
If $x_1=0$ and $x_2=c$, $y(x)$ takes the form 
\be y(x) = \begin{cases} y_0 & x < c , \cr c & c < x < y_0, \cr  
0  & x > y_0, \end{cases}
\ee
In this case set $z=y_0/c$.  We then have
\begin{equation}
t_1=c^2(2z-1)
\end{equation}
\begin{equation}
t_k=c y_0^k+y_0c^k-c^{k+1}=c^{k+1}(z^k+z-1),
\end{equation}

which gives 
\begin{equation}
\frac{t_k}{t_1^{(k+1)/2}}=\frac{z^k+z-1}{(2z-1)^{(k+1)/2}}.
\end{equation}
This function of $z$ is unimodal for $z>1$ (decreasing, then increasing, as can be seen by replacing $z$ with $w=2z-1$)
so the maximum of $t_k$ for fixed $t_1$ occurs at either $y_0=1$ or $y_0=c$, resulting in an anticlique or clique.
\end{itemize}

Thus $y$ cannot take \emph{any} values strictly between $1$ and $c$.}

{Summarizing this section, we have shown that} the path $\gamma_0$ from $(0,1)$ to $(c,c)$ must be first horizontal, then
vertical, and then horizontal, 
although one of the horizontal segments can have zero  length. 

This implies
that $g$ is either bipodal (a clique or anticlique) or tripodal, and if tripodal it is of the form
\be\label{c3=1}g(x,y) = \left\{\begin{array}{ll}1&x<c_1\\c_2&c_1<x<c_2\\c_1&c_2<x<1.\end{array}\right.\ee

\subsubsection{Step 3: Showing that $g$ is bipodal}

If $g$ has the form (\ref{c3=1}),
the edge and $k$-star densities are:
\begin{equation} t_1 = c_1 + (c_2-c_1)c_2 + (1-c_2)c_1 = 2c_1-2c_1c_2+c_2^2 \qquad 
t_k = c_1 + (c_2-c_1)c_2^k + (1-c_2)c_1^k.
\end{equation}
Taking derivatives with respect to $c_j$ gives:
\begin{eqnarray}\label{dotet} 
&\partial_1 t_1 = 2(1-c_2);  & \partial_2 t_1=2(c_2-c_1); \cr & \partial_1 t_k = 1-c_2^k+k(1-c_2)c_1^{k-1}; &
\partial_2 t_k = (c_2^k-c_1^k)+k(c_2-c_1)c_2^{k-1}. 
\end{eqnarray}

An important quantity is the ratio $r_j = 2 \frac{\partial_j t_k}{\partial_j t_1}$. For $j=1,2$, this works out to
\begin{eqnarray}\label{rvalues}
r_1 & = & kc_1^{k-1} + \frac{{1} - c_2^k}{{1}-c_2} \cr 
r_2 & = &  kc_2^{k-1} + \frac{c_2^k - c_1^k}{c_2-c_1} \end{eqnarray}

We imagine $c_1$ and $c_2$ evolving in time so as to keep $t_1$ fixed, with 
$\dot c_1 = c_2-c_1,~~\dot c_2 = c_2 -
{1}$ (note that this
satisfies $\dot t_1=0$ by (\ref{dotet})). We must have $r_1=r_2$, or else $\dot t_k$ would be nonzero. 
We will show that then $d(r_1-r_2)/dt > 0$, and hence that $t_k$ increases to second order. 
We compute:
\begin{eqnarray} 
\dot r_1 - \dot r_2 & = & k(k-1)[c_1^{k-2} \dot c_1 - c_2^{k-2} \dot
  c_2] \cr
&& + k \left [ -\frac{c_1^{k-1}\dot c_2}{{1}-c_2} -
  \frac{c_2^{k-1}\dot c_2}{c_2-c_1} + \frac{c_1^{k-1}\dot
    c_1}{c_2-c_1} \right ] \cr 
&&+ \frac{c_2^k-c_1^k}{c_2-c_1} \left [ \frac{\dot c_2}{{1}-c_2} +
  \frac{\dot c_2}{c_2-c_1} - \frac{\dot c_1}{c_2-c_1}
\right ]
\end{eqnarray}

Plugging in the values of $\dot c_1$ and $\dot c_2$ then gives
\begin{eqnarray}\label{whoops}
\dot r_1 - \dot r_2 & = & k(k-1) \left[ c_1^{k-2}(c_2-c_1) +
  c_2^{k-2}(1-c_2)\right ] \cr 
&& + k \left [ 2c_1^{k-1} + c_2^{k-1} \frac{1-c_2}{c_2-c_1} \right ]
\cr 
&& - \frac{c_2^k-c_1^k}{c_2-c_1} \left [ 2 +
  \frac{1-c_2}{c_2-c_1}\right ]
\end{eqnarray}
Let $x=c_1/c_2$ and let $z=1/c_2$. Note that $x<1<z$. Then $\dot r_1
- \dot r_2$ is  $c_2^{k-1}$ times
\begin{eqnarray} f(x,z) &:= &k(k-1)[x^{k-2}(1-x) + z-1] + k[2x^{k-1} +
    \frac{z-1}{1-x}] - \left ( \frac{1-x^k}{1-x}\right )
\left [ 2 + \frac{z-1}{1-x} \right ] 
\cr 
&=& k(k-1)[x^{k-2}(1-x) + z-1] + 2[(k-1)x^{k-1}-x^{k-2} - \cdots -1]
\cr 
&& +(z-1)[(k-1) + (k-2)x + \cdots + x^{k-2}].
\end{eqnarray}
This is an order $k-1$ polynomial in $x$ and linear in $z$. Likewise,
the constraint $r_1=r_2$ becomes 
\begin{equation}
 0 = F(x,z) : = kx^{k-1} + \frac{z^k-1}{z-1} -k -
\frac{1-x^k}{1-x}.
\end{equation}

Note that $F(1,1) =f(1,1)=0$.  By implicitly differentiating $F$ we
see that $dx/dz = -1$ at $(x,z)=(1,1)$, hence
that $df/dz|_{(1,1)} = \frac32k(k-1)>0$.  Hence $\dot r_1
- \dot r_2$ is positive when $x$ is slightly less than 1. 
If $f(x,z)$ ever fails to be positive on the set where $F(x,z)=0$ then, by the intermediate value
theorem, there is a value of $z$, and a corresponding value of $x$,
for which $f(x,z)=0$ and $F(x,z)=0$.  The intersection
of the two degree $k-1$ algebraic curves
$F=0$ and $f=0$ would then have to contain a real point $(x,z)$ with
$0<x<1$.  

This is easy to check. Solving $f(x,z)=0$ for $z$, we convert $F(x,z)$ into a function of $x$ alone. Plotting this function for $2 \le k \le 30$ and $0 < x < 1$ shows that the function is always negative, approaching a simple zero at $x=1$; see Fig.~\ref{FIG:F-k} for plots of $F(x,z(x))$ as a function of $x$ for some $k$ values. (We were unable to prove this for all $k$; checking larger
values of $k$ is straightforward, but we stopped at 30.) This eliminates the
case (\ref{c3=1}) for $k \le 30$.


We conclude this section with a qualitative feature of phase diagrams 
for models with constraints on edges and any other simple graph $H$.

\begin{theorem}\label{simplyconnected} The phase diagram  for an edge-$H$ model is simply
  connected.
\end{theorem}

\begin{proof} This is just the intermediate value theorem. For each
  fixed edge density $\E$, let $g_\E$ be a graphon that maximizes
  $t_H$ and $h_\E$ a graphon which minimizes $t_H$. Consider the family of graphons 
$m_{\E,a}(x,y) = a h_\E(x,y) + (1-a) g_\E(x,y).$ When $a=0$ we get the maximal
  value of $t_H$, when $a=1$ we get the minimal value, and by
  continuity we have to get all values in between.  In other words,
  the entire phase space between the upper boundary and the lower boundary
  is filled in. 
\end{proof}

Note: This is an application of a technique we learned from Oleg Pikhurko
\cite{P}.

\begin{figure}[!ht]
\centering
\includegraphics[width=3in]{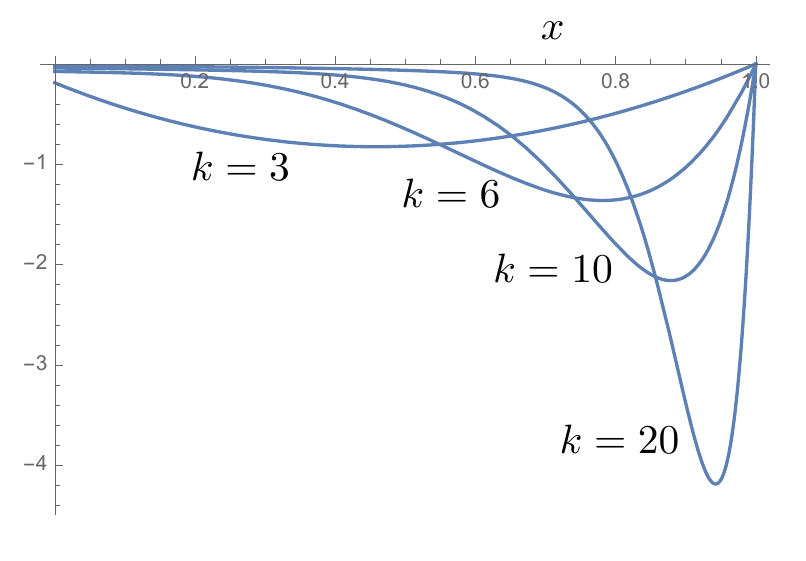} 
\caption{Plots of $F(x,z(x))$ as a function of $x$ for $k=3$, $k=6$, $k=10$ and $k=20$.}
\label{FIG:F-k}
\end{figure}

\section{Phase Transition for $2$-Stars}
\label{SEC: trans}

\begin{theorem}\label{phase_theorem}
For the 2-star model 
there are inequivalent graphons maximizing the constrained entropy on the line
segment $\{(1/2,\Tt)\,|\,\T^*< \Tt\le 2^{-3/2}\}$ for some $\T^*< 2^{-3/2}$. {Moreover,
near $(1/2,2^{-3/2})$, the maximizing graphons do not vary continuously with the constraint parameters.}
\end{theorem}


\begin{proof}

For any graphon $g$, consider the graphon $g'(x,y) = 1-g(x,y)$. The degree
functions for $g'$ and $g$ are related by $d'(x) = 1-d(x)$. If $g$ has
edge and 2-star densities $\E$ and $\Tt$, then $g'$ has edge density $1-\E$ 
and 2-star density $\int_0^1 (1-d(x))^2 dx = 1-2\E+\Tt$. Furthermore, 
$\S(g')=\S(g)$. This implies that $g'$ maximizes the entropy at $(1-\E, 
1-2\E+\Tt)$ if and only if $g$ maximizes the entropy at $(\E,\Tt)$. 
In particular, if $g$ maximizes the entropy at $(1/2, \Tt)$, then so
does $g'$. To show that $\S(g)$ has a non-unique maximizer along the upper
part of the $\E=1/2$ line, we must only show that a maximizer $g$ is not related to its mirror $g'$ 
by a measure-preserving transformation of $[0,1]$. 

As noted in Section~\ref{SEC: space}, up to such a transformation
there are exactly two graphons corresponding
to $(\E,\Tt)=(1/2, 1/({2\sqrt{2}}))$, namely a g-anticlique $g_a$ and a
g-clique $g_c$. These are not related by reordering, since the values of the degree
function for the g-clique are $\sqrt{2}/2$ and 0, while those for the 
g-anticlique are $1$ and $1 - \sqrt{2}/2$.  Let $D$ be smallest of the
following distances in the cut metric (see Chapter 8 in \cite{Lov}): 
(1) from $g_a$ to $g_c$,
(2) from $g_a$ to the set of symmetric graphons, and (3) from $g_c$ to the set of 
symmetric graphons. 

\begin{lemma}There exists $\delta>0$ such that every graphon with $(\E,\Tt)$ within $\delta$ of 
$(1/2, 2^{-3/2})$ is within $D/3$ of either $g_a$ or $g_c$. 
\end{lemma}

\begin{proof} Suppose otherwise.  Then we could find a sequence of graphons with $(\E,\Tt)$ converging to $(1/2, 2^{-3/2})$ that
have neither $g_a$ nor $g_c$ as an accumulation point. However, the space of reduced graphons  
is known to be compact \cite{Lov}, so there must
be some accumulation point $g_\infty$ that is neither $g_a$ nor $g_c$. Since convergence in the cut metric implies convergence of the density of all subgraphs, 
$t_1(g_\infty)=1/2$ and $t_2(g_\infty) = 2^{-3/2}$. But this contradicts the fact that only $g_a$ and $g_ c$ have edge and 2-star
densities $(1/2, 2^{-3/2})$. 
\end{proof}

By the lemma, no graphon with $\E=1/2$
and $\Tt > 1/2\sqrt{2}- \delta$ is invariant 
(up to reordering) under $g \to 1-g$.  In particular, the entropy maximizers cannot be symmetric, so there must be two (or more)
entropy maximizers, one close to $g_a$ and one close to $g_c$.

Moreover, on a path in the parameter space 
near the upper boundary, from the anticlique on the upper boundary
at $\E=\frac12-\delta$ to the
clique on the upper boundary at $\E=1/2+\delta$, there is a discontinuity in the graphon, where it jumps from being close to $g_a$ to being
close to $g_c$. There must be an odd number of such jumps, and if the path is chosen to be symmetric with respect to
the transformation $\E \to 1-\E$, $\T_2 \to \T_2 + 1-2\E$, the jump points must be arranged symmetrically on the path.
In particular, one of the jumps must be at exactly $\E=1/2$. This shows that the $\E=1/2$ line forms the boundary between
a region where the optimal graphon is close to $g_a$ and another region where the optimal graphon is close to $g_c$.
\end{proof}

\section{Simulations}
\label{SEC:simul}
We now show some numerical simulations in the 2-star model ($\ell=2,
k_1=1, k_2=2$). Our main aim here is to present numerical evidence
that the maximizing graphons in this case are in fact \emph{bipodal},
and to clarify the significance of the degeneracy of 
Theorem~\ref{phase_theorem}.

To find maximizing $K$-podal graphons, we partition the interval $[0, 1]$ into $K$ subintervals $\{I_i\}_{i=1,\dots,k}$ with lengths $c_1,c_2,\cdots,c_K$,
that is, $I_i=[c_0+\dots+c_{i-1},c_0+\dots+c_i]$ (with $c_0=0$). We form a partition of the square $[0,1]^2$ using the product of this partition with itself. 
We are interested in functions $g$ that are piecewise constant on the partition:
\begin{equation}
	g(x,y)=g_{ij},\ \ (x,y)\in I_i\times I_j, \quad 1\le i,j\le K,
\end{equation}
with $g_{ij}=g_{ji}$. We can then verify that the entropy density $\S(g)$, the edge density $t_1(g)$ and the 2-star density $t_2(g)$ become respectively
\begin{equation}
	\S(g)=-\dfrac{1}{2} \dsum_{1\le i,j\le K} [g_{ij}\log g_{ij}+(1-g_{ij})\log(1-g_{ij})]c_i c_j,
\end{equation}
\begin{equation}
	t_1(g) = \dsum_{1\le i, j\le K} g_{ij}c_i c_j,\qquad t_2(g)=\dsum_{1\le i,j,k\le K} g_{ik}g_{kj}c_i c_j.
\end{equation}

Our objective is to solve the following maximization problem:
\begin{equation}
\max_{\{c_j\}_{1\le j\le K}, \{g_{i,j}\}_{1\le i,j\le K}} \S(g), \quad \mbox{subject to:}\quad t_1(g)=\epsilon, \quad t_2(g)=\tau_2, \quad \dsum_{1\le j\le K}c_j=1, \quad g_{ij}=g_{ji}.
\end{equation}

We developed in~\cite{RRS} computational algorithms for solving this
maximization problem
and have benchmarked the algorithms with theoretically known results. For a fixed $\tau\equiv (\E,\Tt)$, our strategy is to first maximize for a fixed number $K$, and then maximize over the number $K$. Let $s_{(\E,\Tt)}^K$ be the maximum achieved by the graphon $g_{_K}$, then the maximum of the original problem is $s_{(\E,\Tt)}=\max_K\{s_{(\E,\Tt)}^K\}$. Our computational resources allow us to go up to $K=16$ at this time. See~\cite{RRS} for more details on the algorithms and their benchmark with existing results.

The most important numerical finding in this work is that, for every pair $(\E,\Tt)$ in the interior of the phase space, the graphons that maximize $\S(g)$ are \emph{bipodal}. We need only four parameters ($c_1$, $g_{11}$, $g_{12}$ and $g_{22}$) to describe bipodal graphons (due to the fact that $c_2=1-c_1$ and $g_{12}=g_{21}$). For maximizing bipodal graphons, we need only three parameters, since (\ref{Euler-Lagrange}) implies that
\begin{equation}\label{EQ:Maxi Graphon}
\left(\frac1{g_{11}}-1\right)\left(\frac1{g_{22}}-1\right)=\left(\frac1{g_{12}}-1\right)^2,
\end{equation}
which was used in our numerical algorithms to simplify the calculations.

We show in Fig.~\ref{FIG:Graphons} maximizing graphons at some typical points in the phase space. The $(\E,\Tt)$ pairs for the plots are respectively: $(0.3,0.16844286)$ and $(0.3,0.10339268)$ for the first column (top to bottom), $(0.5,0.32455844)$ and $(0.5,0.27485281)$ for the second column, and $(0.7,0.56270313)$ and $(0.7,0.50339268)$ for the third column.
\begin{figure}[!ht]
\centering
\includegraphics[angle=0,width=0.23\textwidth]{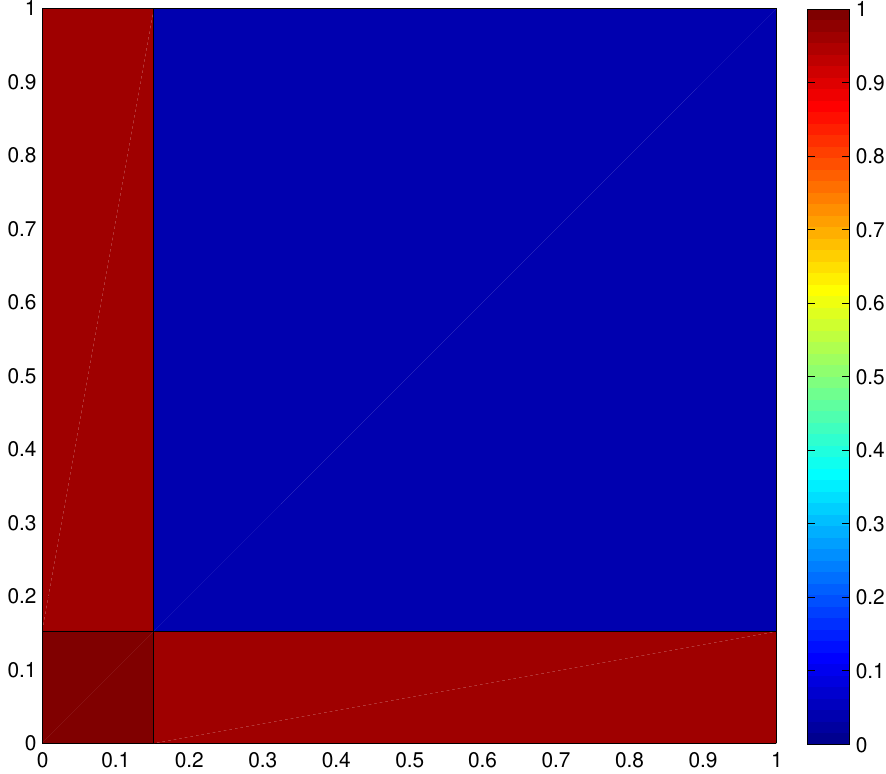} 
\includegraphics[angle=0,width=0.23\textwidth]{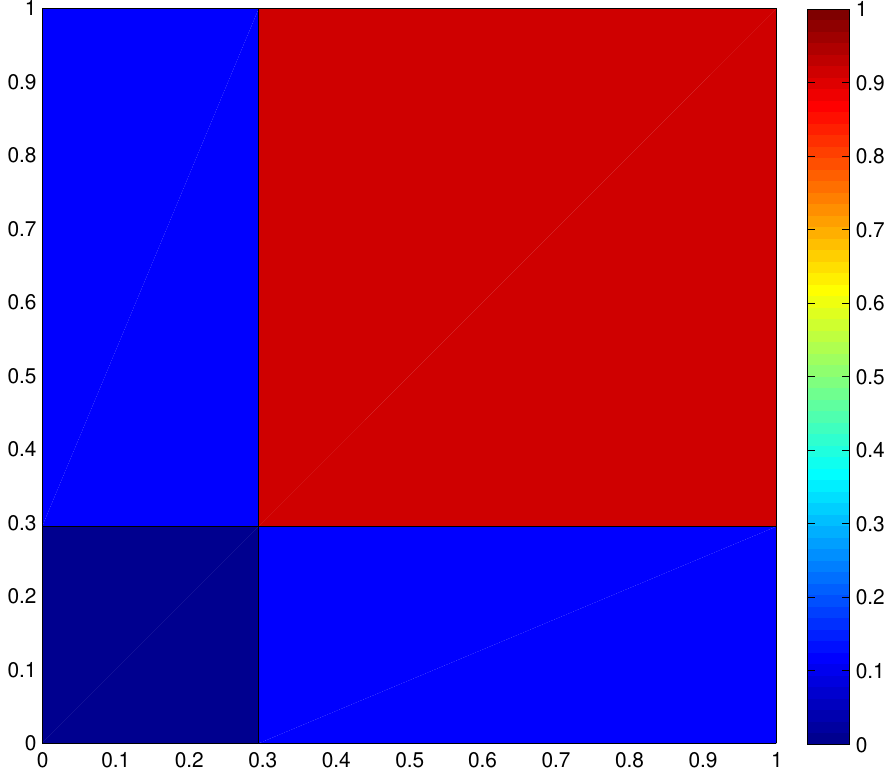}
\includegraphics[angle=0,width=0.23\textwidth]{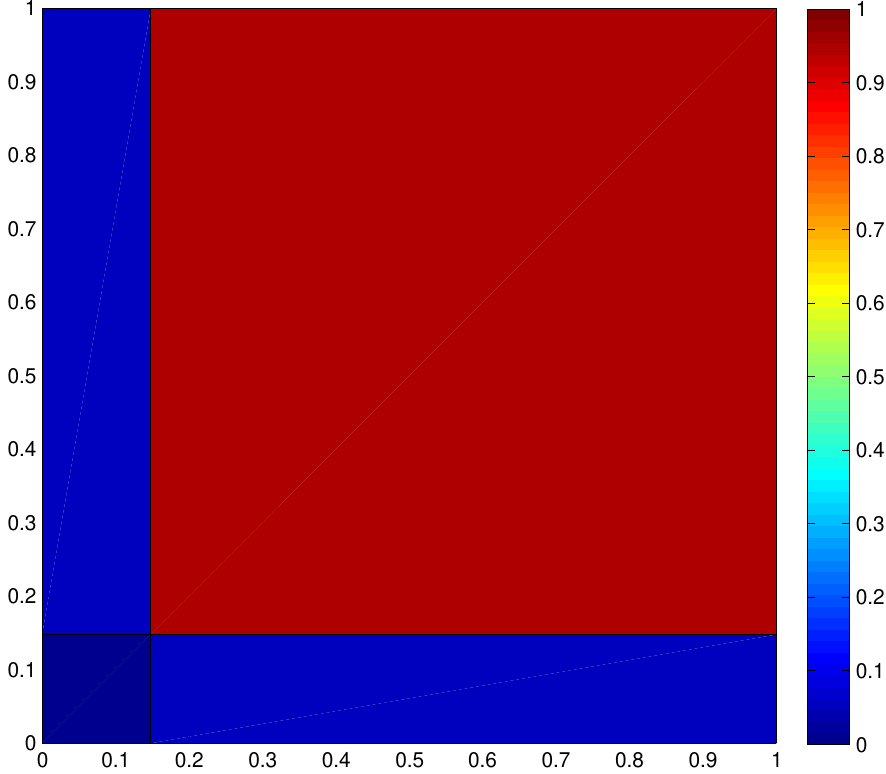}\\
\includegraphics[angle=0,width=0.23\textwidth]{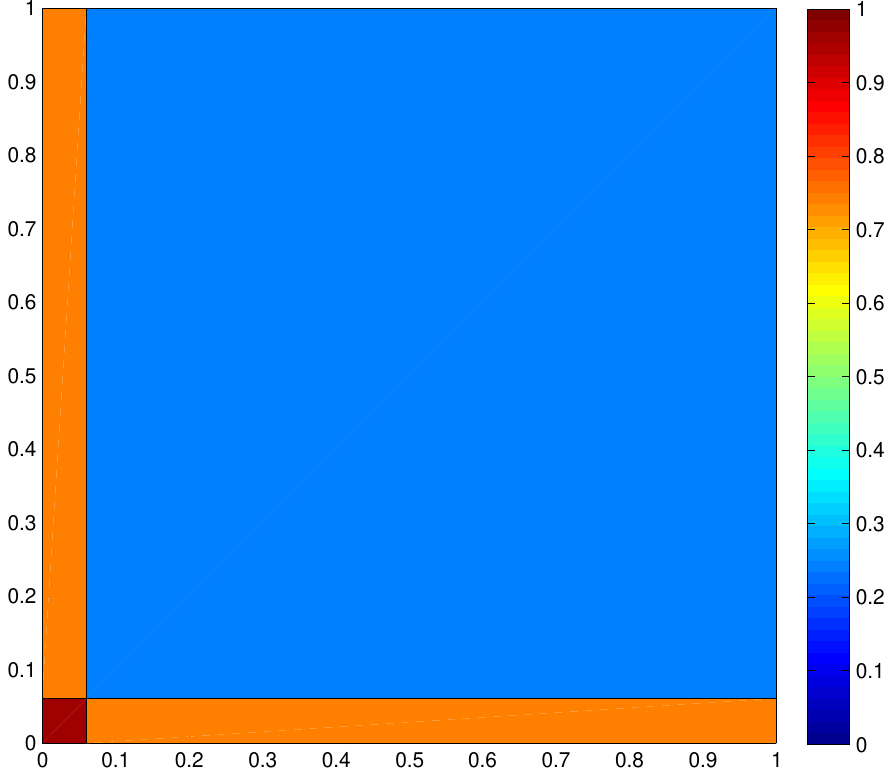}
\includegraphics[angle=0,width=0.23\textwidth]{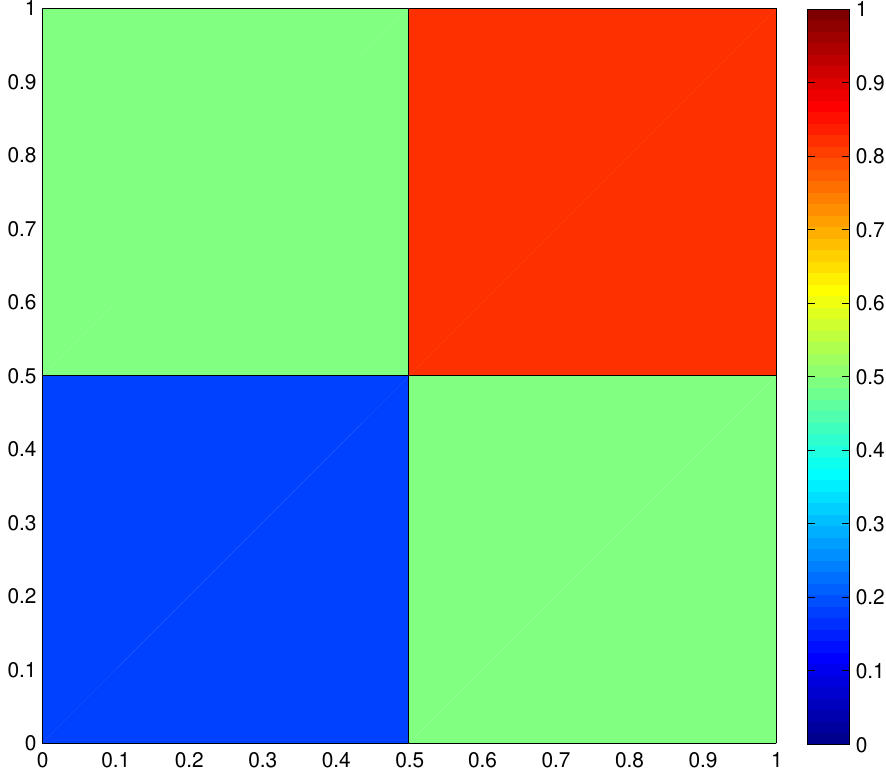}
\includegraphics[angle=0,width=0.23\textwidth]{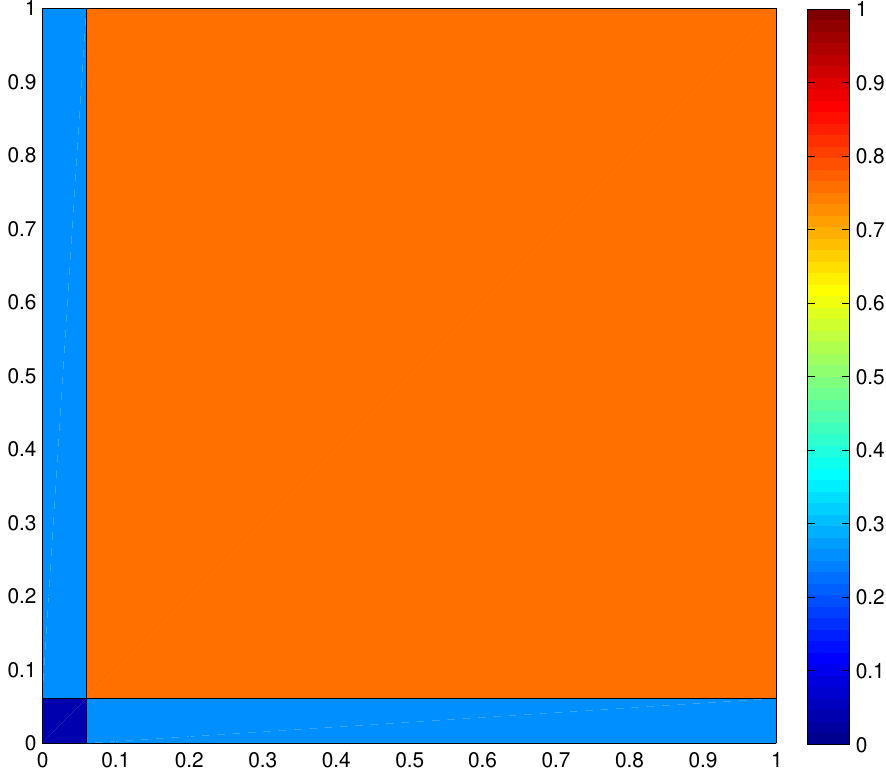}
\caption{Maximizing graphons at $\E=0.3$ (left column),
  $\E=0.5$ (middle column) and $\E=0.7$ (right column). For each
  column $\Tt$ values decrease from top to bottom.}
\label{FIG:Graphons}
\end{figure}

The values of $s$ corresponding to the maximizing graphons are shown in the left plot of Fig.~\ref{FIG:Phase-S} for a fine grid of $(\E,\sigma^2)$ (with $\sigma^2=\Tt-\E^2$ as defined in Fig.~\ref{FIG:Phase-Boundary}) pairs in the phase space. We first observe that the plot is symmetric with respect to $\E=1/2$. The symmetry comes from the fact {(see the proof of Theorem \ref{phase_theorem})} 
that the map $g\to 1-g$ takes $\E\to 1-\E$, $\Tt\to 1-2\E+\Tt$ and thus $\sigma^2\to \sigma^2$. To visualize the landscape of $s$ better in the phase space, we also show the cross-sections of $s_{(\E,\Tt)}(\E,\sigma^2)$ along the lines $\E_k=0.05 k$, $k=7,\cdots,13$, in the right plots of Fig.~\ref{FIG:Phase-S}.
\begin{figure}[!ht]
\begin{minipage}{0.49\textwidth}
\begin{center}
\includegraphics[angle=0,width=0.98\textwidth]{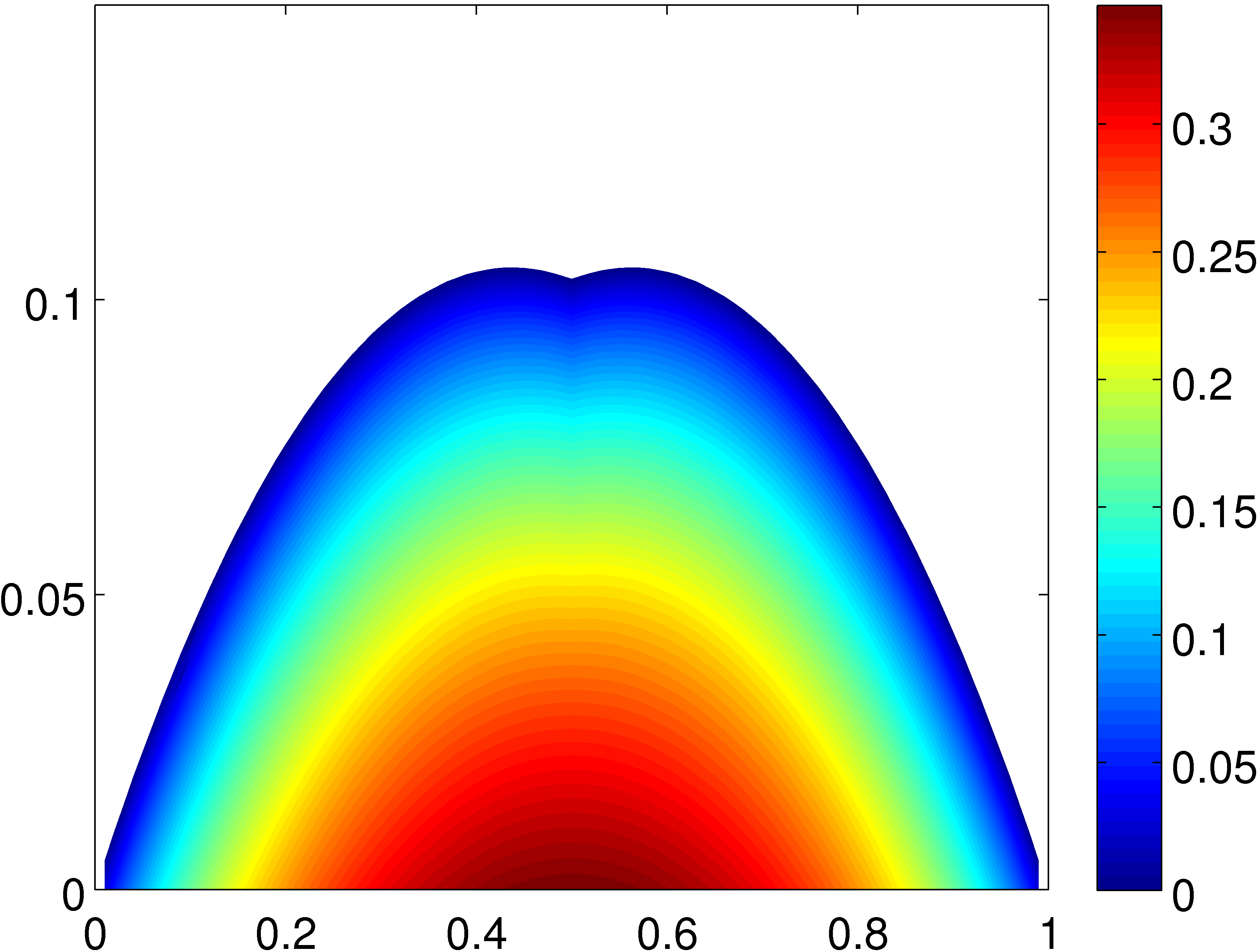}
\end{center}
\end{minipage}
\begin{minipage}{0.49\textwidth}
\begin{center}
\includegraphics[angle=0,width=0.32\textwidth]{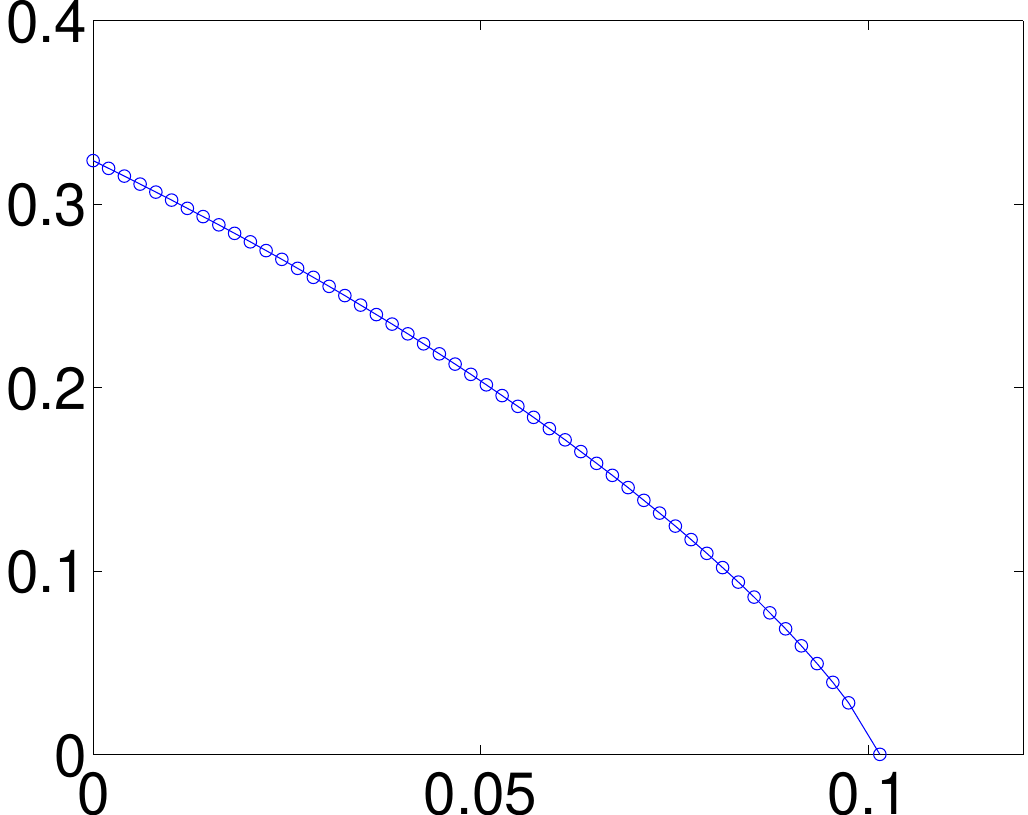} 
\includegraphics[angle=0,width=0.32\textwidth]{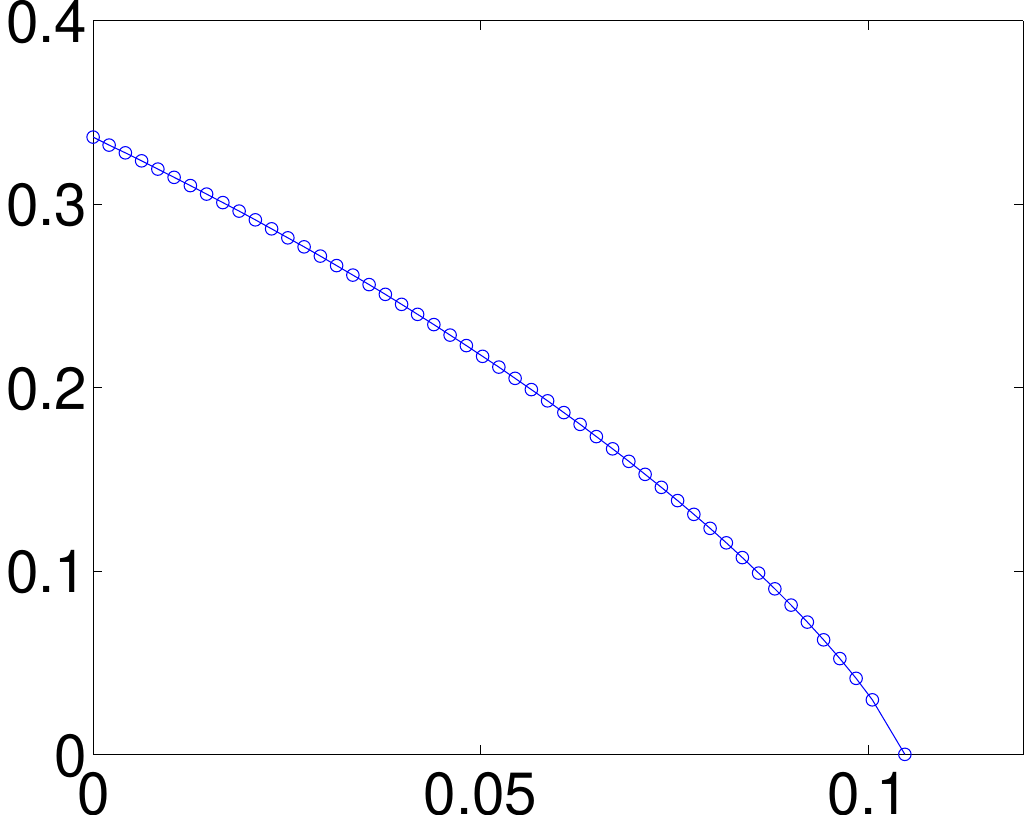} 
\includegraphics[angle=0,width=0.32\textwidth]{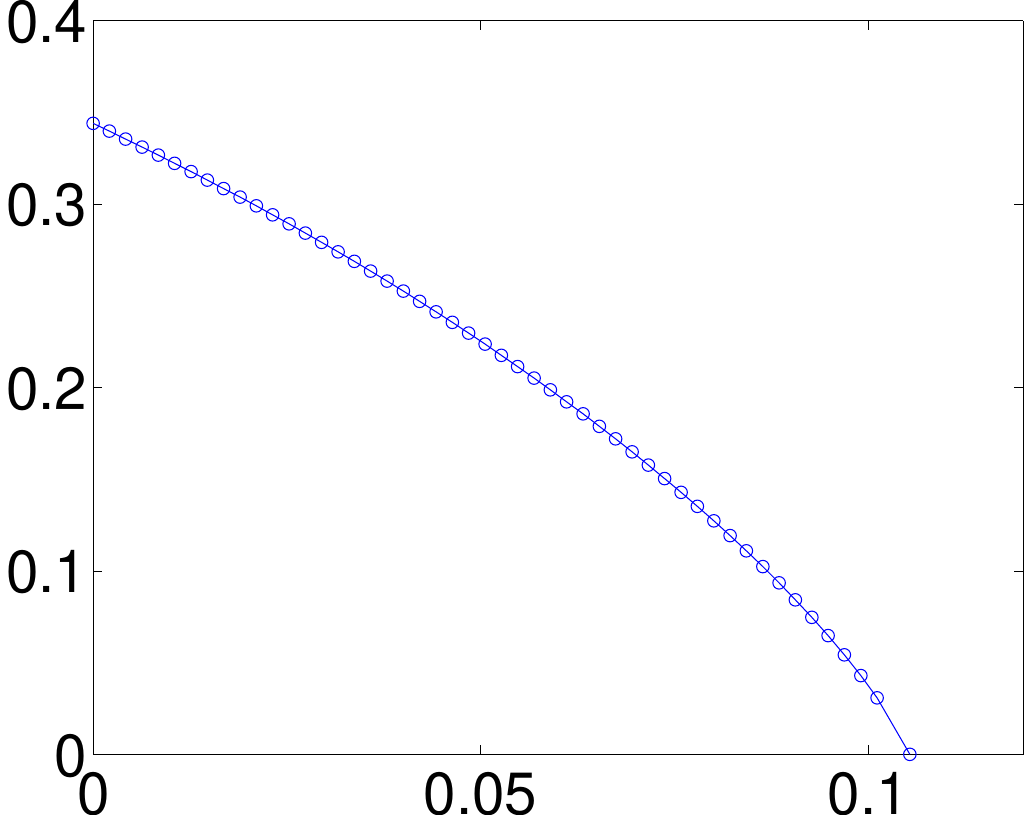}\\ 
\includegraphics[angle=0,width=0.32\textwidth]{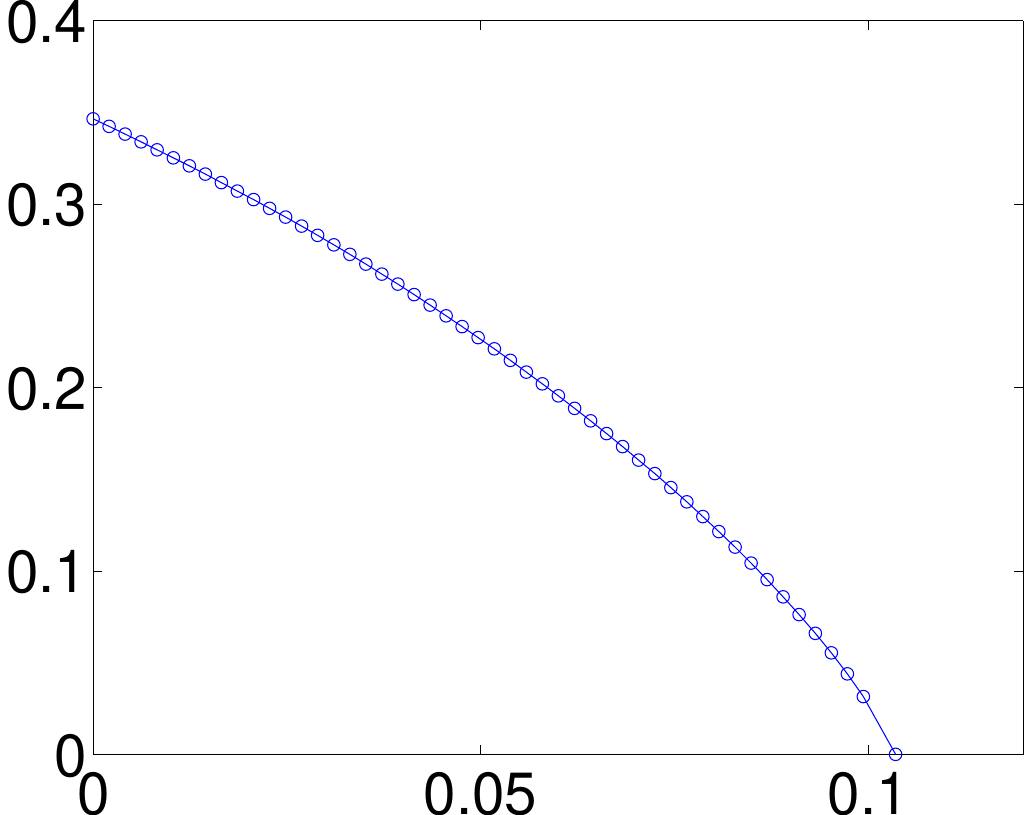}\\
\includegraphics[angle=0,width=0.32\textwidth]{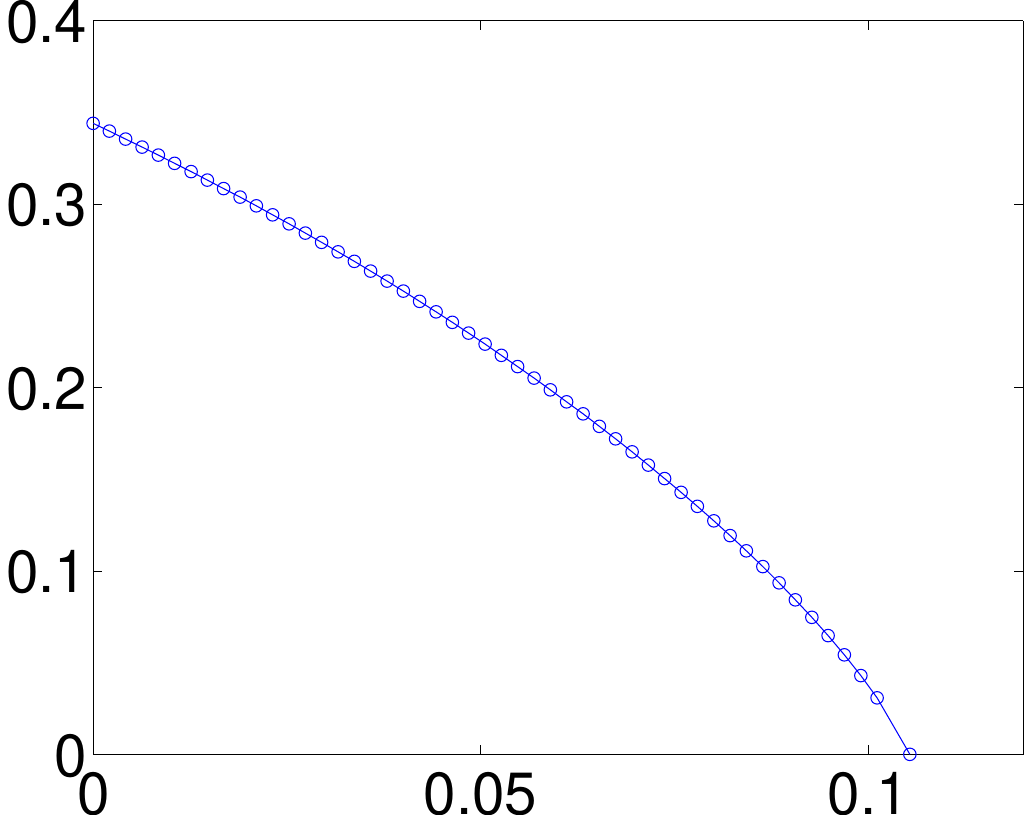} 
\includegraphics[angle=0,width=0.32\textwidth]{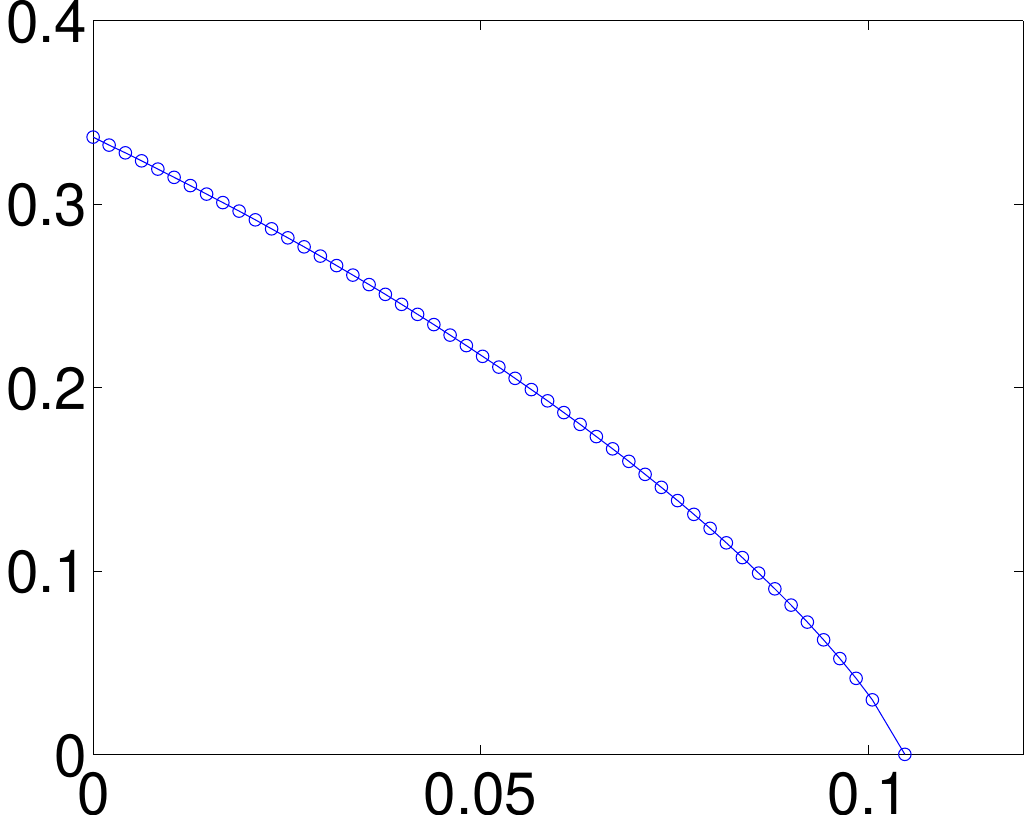} 
\includegraphics[angle=0,width=0.32\textwidth]{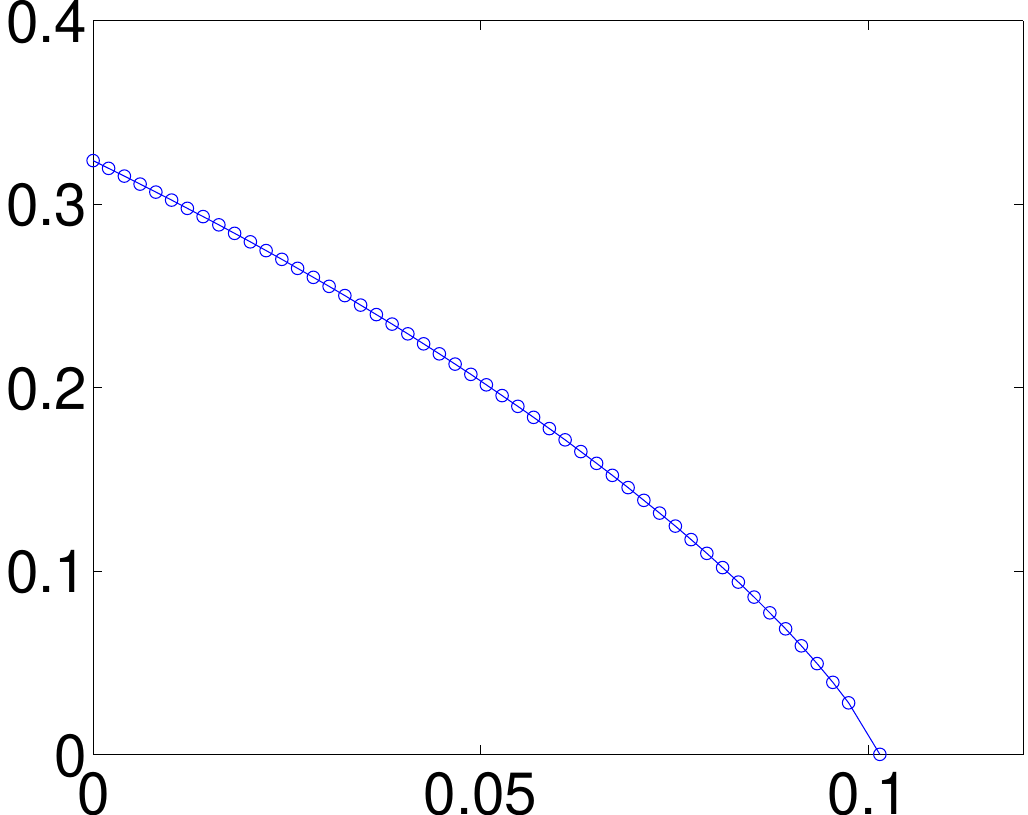} 
\end{center}
\end{minipage}
\put(-464,7){${\boldsymbol\sigma^2}$}
\put(-361,-88){${\boldsymbol\E}$}
\caption{Left: values of $s_{(\E,\Tt)}$ at different $(\E,\sigma^2)$ pairs; Right: cross-sections of $s_{(\E,\Tt)}(\E,\sigma^2)$ along lines $\E=\E_k=0.05 k$ ($k=7,\cdots, 13$) (from top left to bottom right).}
\label{FIG:Phase-S}
\end{figure}

We show in the left plot of Fig.~\ref{FIG:Phase-C} the values of $c_1$ of the maximizing graphons as a function of the pair $(\E,\sigma^2)$. Here we associate $c_1$ with the set of vertices among $V_1$ and $V_2$ that has the larger probability of an interior edge. This is done to avoid the ambiguity caused by the fact that one can relabel $V_1$, $V_2$ and exchange  $c_1$ and $c_2$ to get an equivalent graphon with the same $\E$, $\Tt$ and $\S$ values. We again observe the symmetry with respect to $\E=1/2$. The cross-sections of $c_1(\E,\sigma^2)$ along the lines of $\E_k=0.05 k$ ($k=7,\cdots, 13$) are shown in the right plots of Fig.~\ref{FIG:Phase-C}.
\begin{figure}[!ht]
\begin{minipage}{0.49\textwidth}
\begin{center}
\includegraphics[angle=0,width=0.98\textwidth]{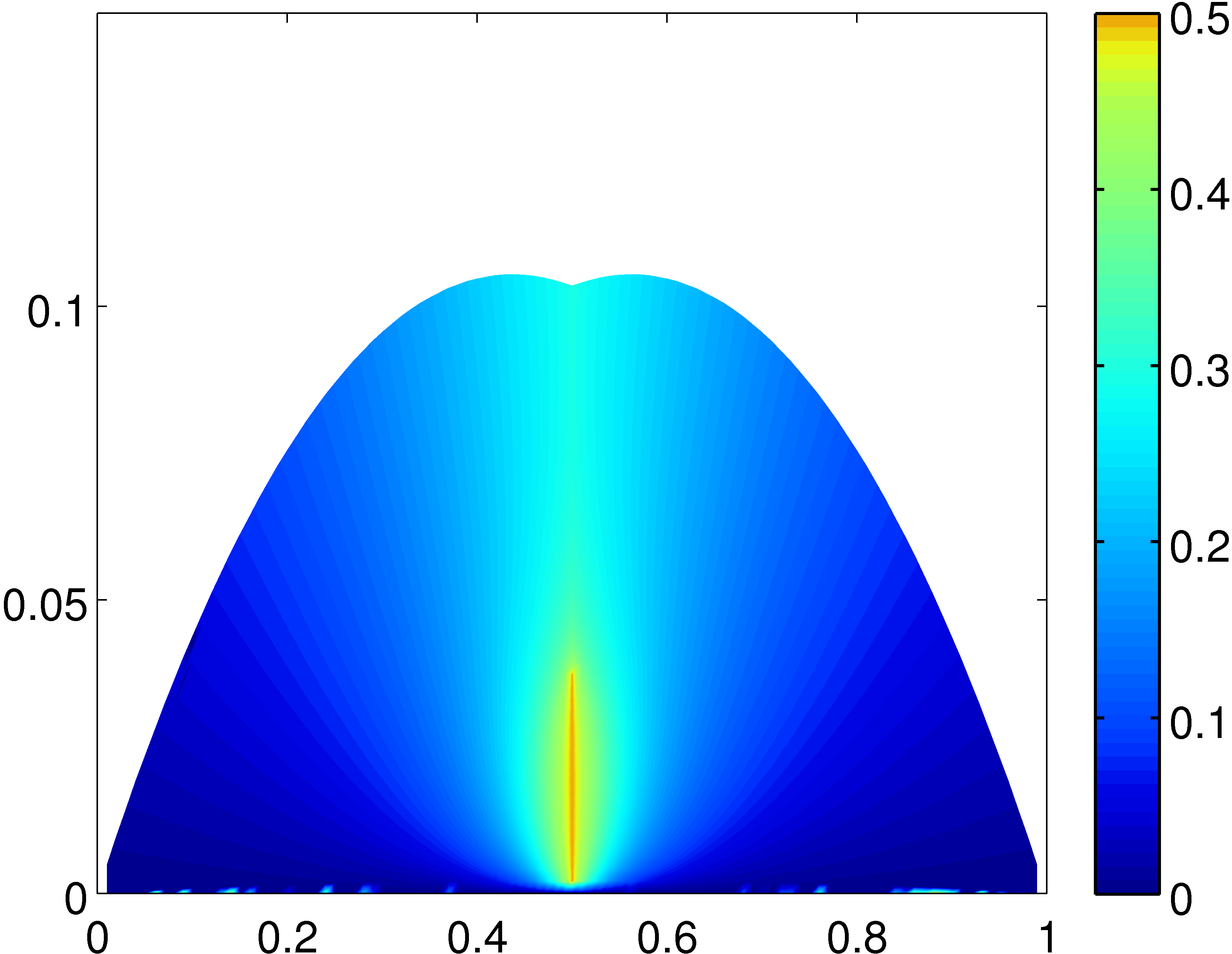}
\end{center}
\end{minipage}
\begin{minipage}{0.49\textwidth}
\begin{center}
\includegraphics[angle=0,width=0.32\textwidth]{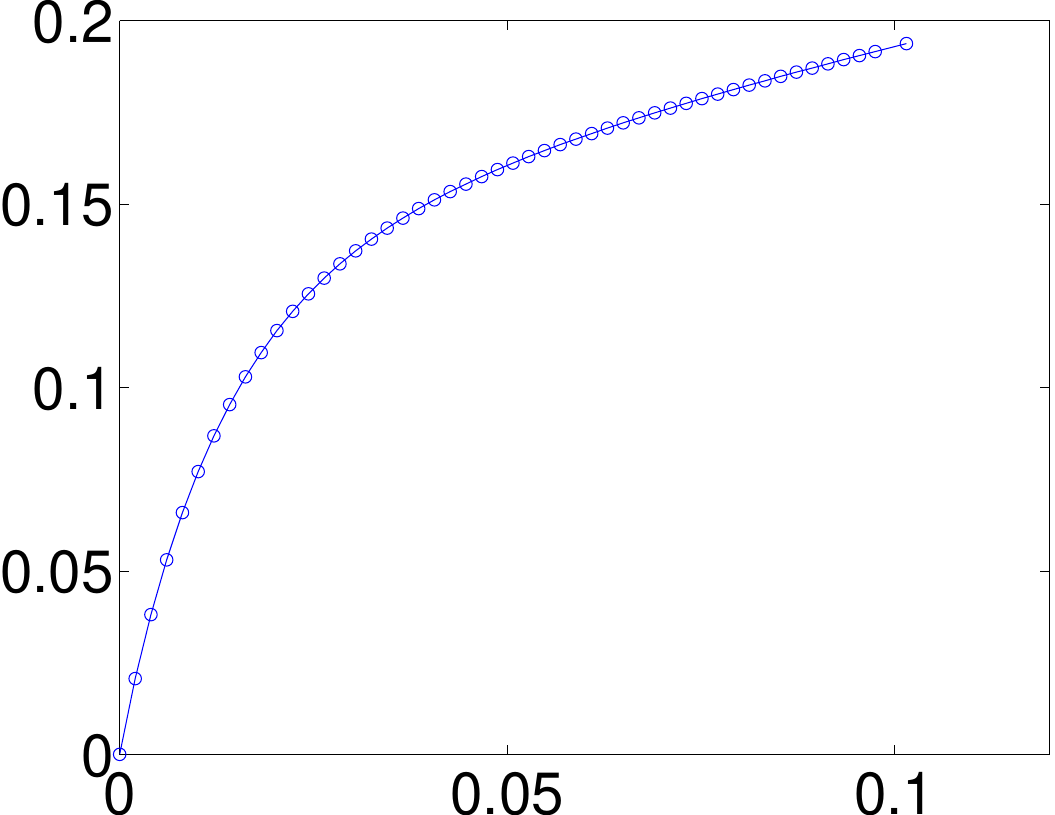} 
\includegraphics[angle=0,width=0.32\textwidth]{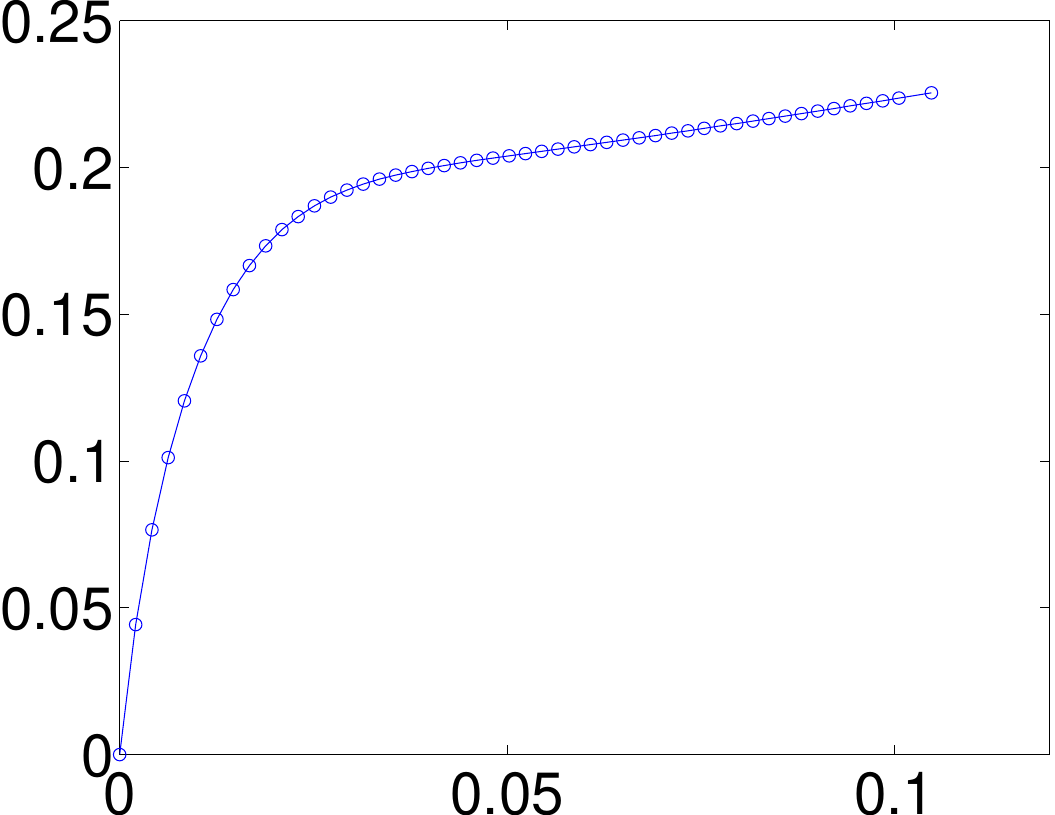} 
\includegraphics[angle=0,width=0.32\textwidth]{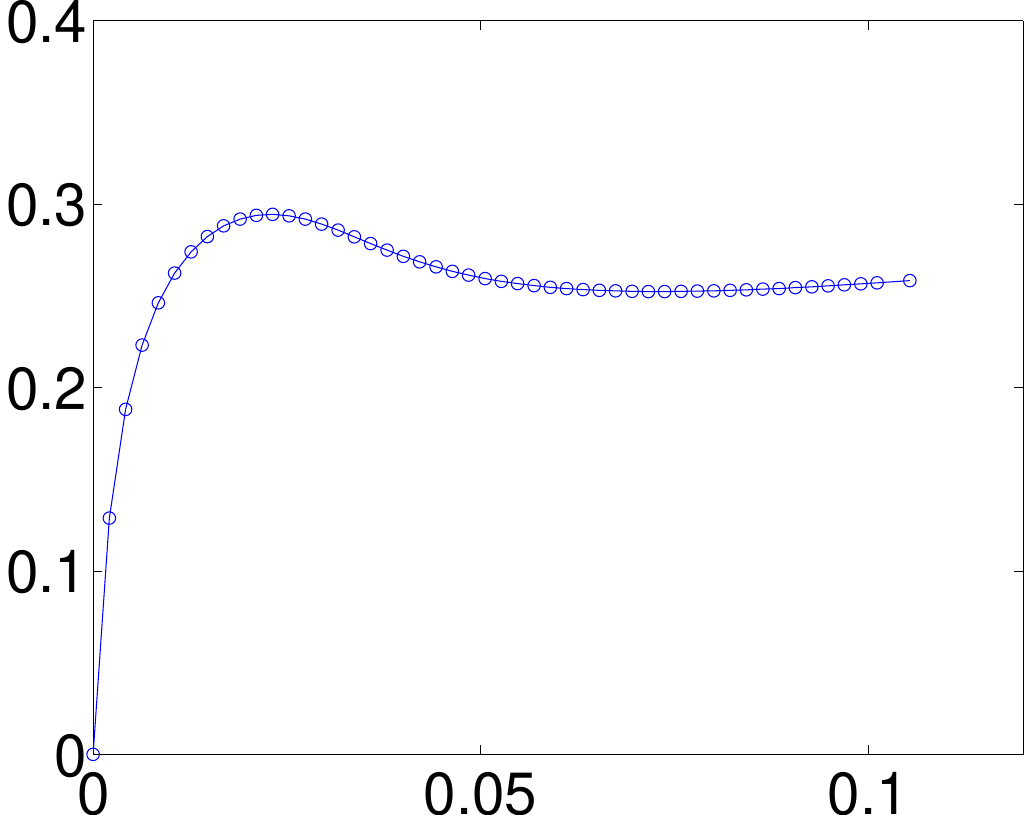}\\ 
\includegraphics[angle=0,width=0.32\textwidth]{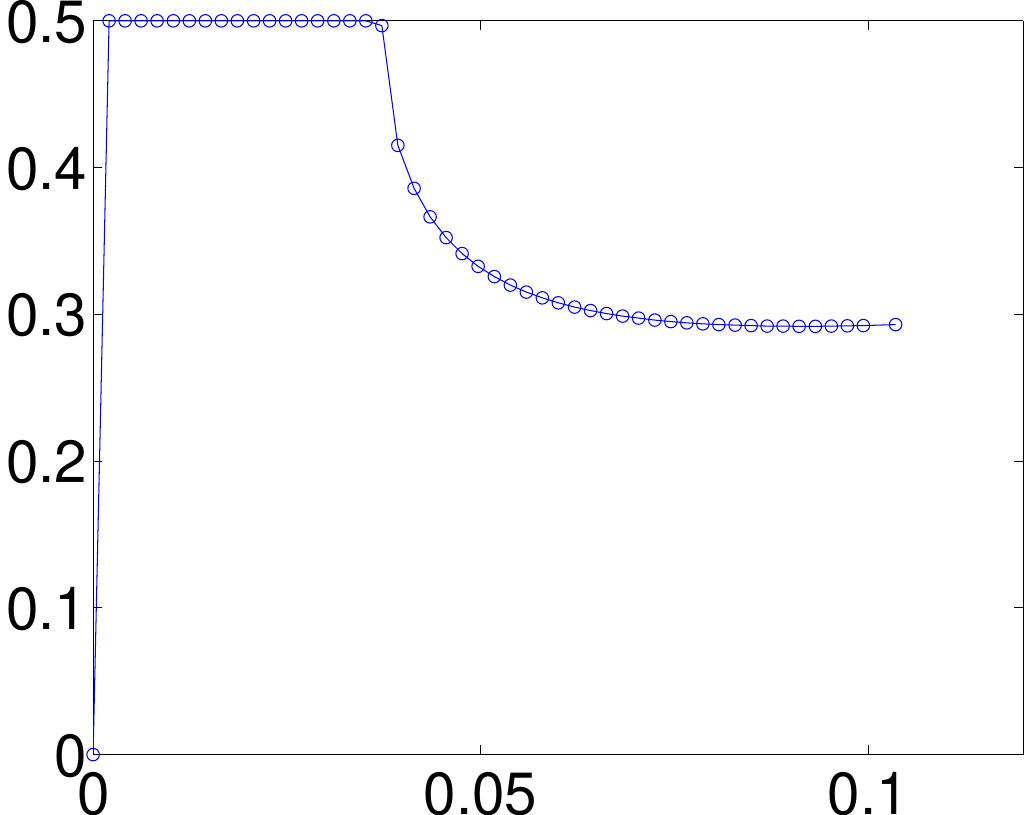}\\
\includegraphics[angle=0,width=0.32\textwidth]{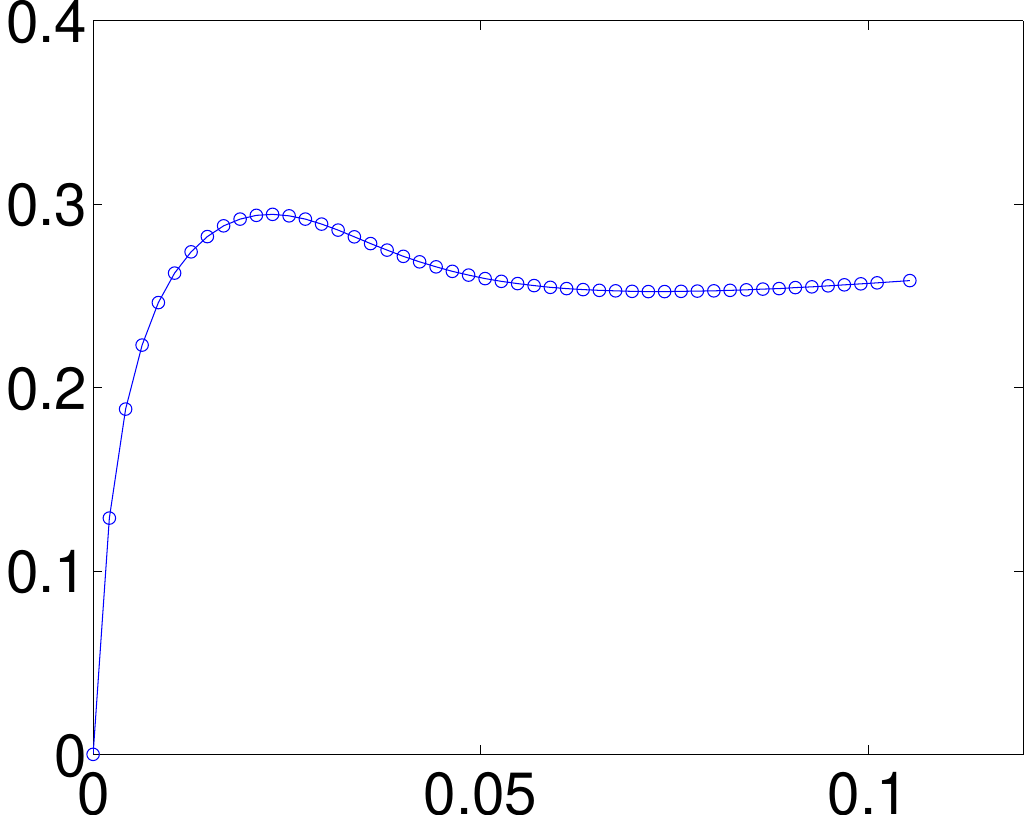} 
\includegraphics[angle=0,width=0.32\textwidth]{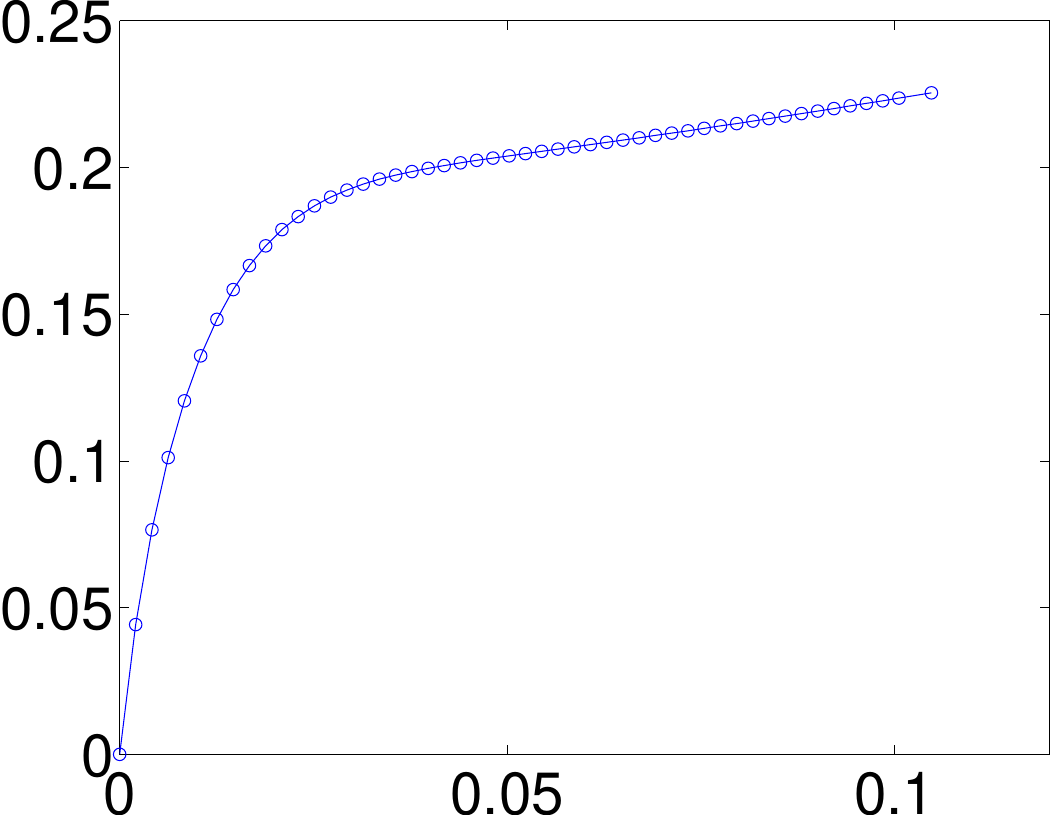} 
\includegraphics[angle=0,width=0.32\textwidth]{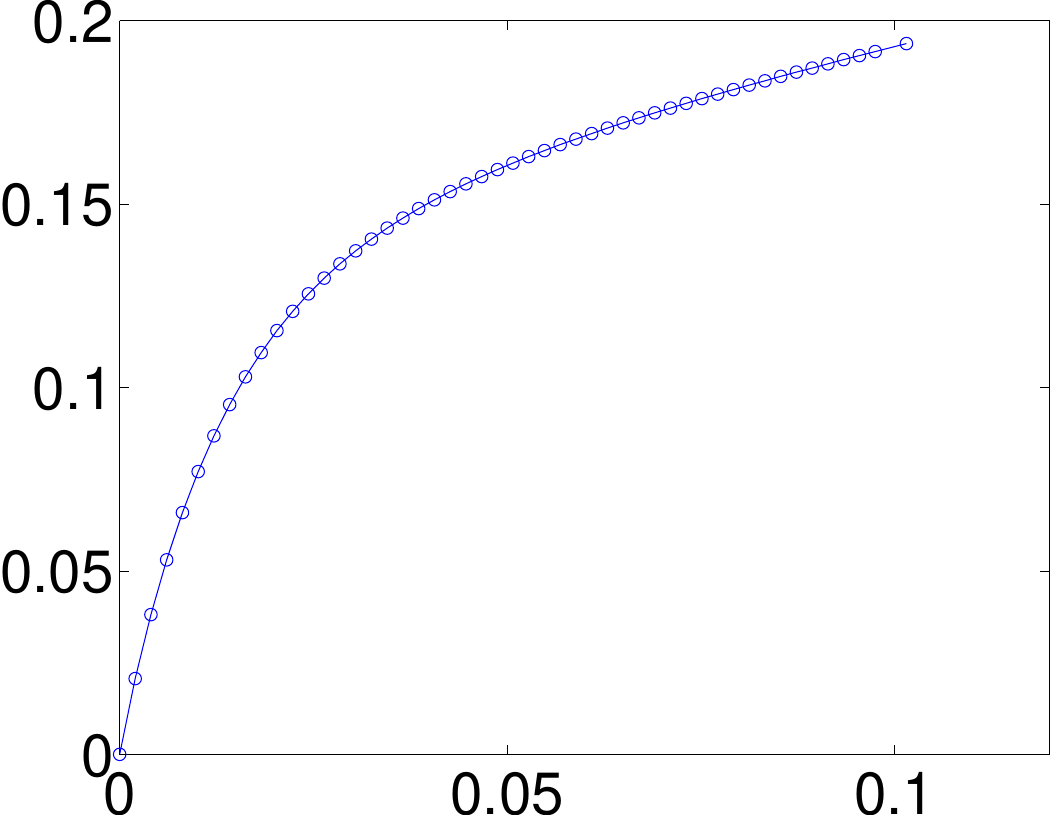} 
\end{center}
\end{minipage}
\put(-464,7){${\boldsymbol\sigma^2}$}
\put(-358,-88){${\boldsymbol\E}$}
\caption{Left: $c_1$ of the maximizing bipodal graphons as a function of $(\E,\sigma^2)$; Right: cross-sections of $c_1(\E,\sigma^2)$ along lines of $\E_k=0.05 k$ ($k=7,\cdots, 13$).}
\label{FIG:Phase-C}
\end{figure}

The last set of numerical simulations were devoted to the study of a
phase transition in the 2-star model. The existence of this phase
transition is suggested by the degeneracy in Theorem~\ref{phase_theorem}. Our numerical simulations indicate that the functions differ to first order in $\E-1/2$, and that the actual entropy $s_{(\E,\Tt)} = \max\{{s^L}_{(\E,\Tt)}, s^R_{(\E,\Tt)}\}$ has a discontinuity in $\partial_\E s_{(\E,\Tt)}$ at $\E=1/2$ above a critical value $\Tt^c$. Below $\Tt^c$, there is a single maximizer, of the form
\begin{equation}\label{symmetric_graphon}
g(x,y) = \begin{cases} \frac{1}{2} + \nu & x,y < \frac{1}{2} \cr
\frac{1}{2} - \nu & x,y > \frac{1}{2} \cr
\frac{1}{2} & \hbox{otherwise} \end{cases}
\end{equation}
Here $\nu$ is a parameter related to $\Tt$ by $\Tt = 1/4 + {\nu^2}/{4}$. Applying the symmetry $g \to 1-g$ and reordering the interval $[0,1]$ by $x \to 1-x$ sends $g$ to itself. 

The critical point $\Tt^c$ is located on the boundary of the region in which the maximizer~\eqref{symmetric_graphon} is stable. The value of $\Tt^c$ can be found by computing the second variation of $\S(g)$ within the space of bipodal graphons with fixed values of $(\E=\frac{1}{2},\Tt)$, evaluated at the maximizer~\eqref{symmetric_graphon}. This second variation is positive-definite for $\nu$ small ({\it i.e.} for $\Tt$ close to $1/4$) and becomes indefinite for larger values of $\nu$. At the critical value of $\Tt^c$, $\nu = 2\sqrt{\Tt^c - 1/4}$ satisfies 
\begin{equation} 
\left ( 2S(\frac{1}{2}-\nu) -2S(\frac{1}{2}) + 3 \nu S'(\frac{1}{2}-\nu) \right ) (2 -\frac{1}{2} S''(\frac{1}{2}-\nu))
+ 8 \nu^2 S''(\frac{1}{2}-\nu) = 0
\end{equation}
where $S'$ and $S''$ are respectively the first and second order derivatives of $S(g)$ (defined in~\eqref{EQ:Shannon}) with respect to $g$. This equation is transcendental, and so cannot be solved in closed form. Solving it numerically for $\nu$ leads to the value $\Tt^c\approx 0.287$, or $\sigma^2 \approx 0.037$. This agrees precisely with what we previously observed in our simulations of optimizing graphons, and corresponds to the point in the left plot Fig.~\ref{FIG:Phase-C} where the $c_1=1/2$ region stops.  

\begin{figure}[ht]
\centering
\includegraphics[angle=0,width=0.48\textwidth]{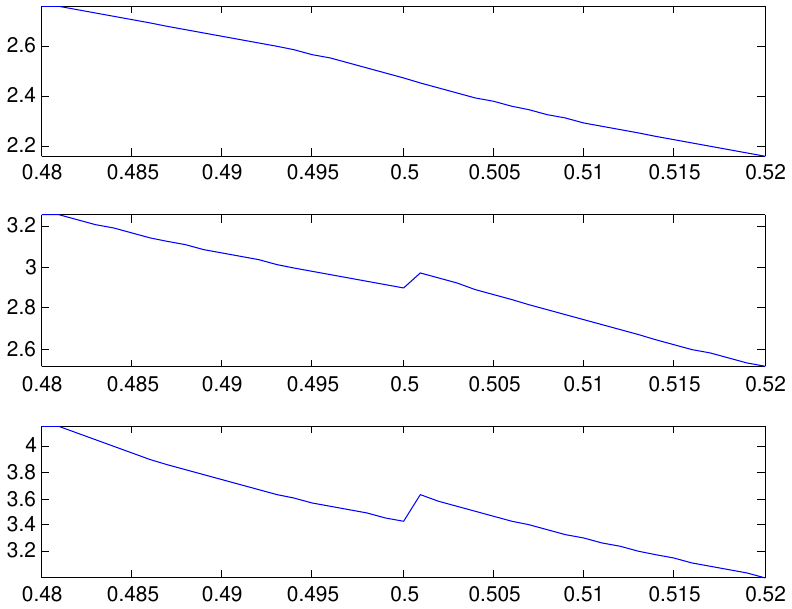} 
\includegraphics[angle=0,width=0.48\textwidth]{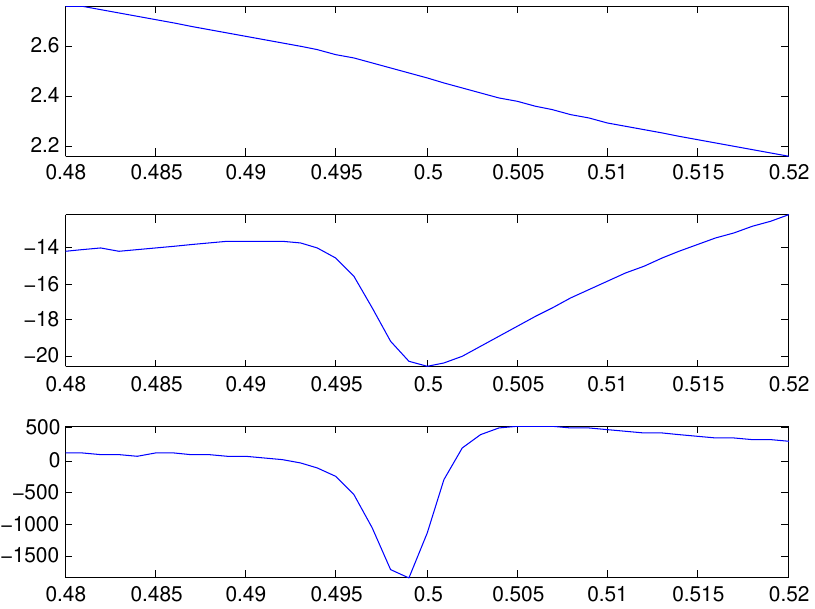} 
\caption{Left: the derivative $\frac{\partial s_{(\E,\Tt)}}{\partial \E}$ at $\Tt=0.28$ (top), $\Tt=0.30$ (middle) and $\Tt=0.32$ (bottom) in the neighborhood of $\E=0.5$; Right: the derivatives $\frac{\partial s_{(\E,\Tt)}}{\partial \E}$ (top), $\frac{\partial^2 s_{(\E,\Tt)}}{\partial \E^2}$ (middle) and $\frac{\partial^3 s_{(\E,\Tt)}}{\partial \E^3}$ (bottom) in the neighborhood of $\E=0.5$ for $\Tt=0.28$.}
\label{FIG:PhaseTrans}
\end{figure}

In the left plot of Fig.~\ref{FIG:PhaseTrans}, we show numerically computed derivatives of $s_{(\E,\Tt)}$ with respect to $\E$ in the neighborhood of $\E=0.5$ for three different values of $\Tt$: one below the critical point and two above it. It is clear that discontinuities in the first order derivative of $s$ appears at $\E=0.5$ for $\Tt>\Tt^c$. When $\Tt<\Tt^c$, we do not observe any discontinuity in the first three derivatives of $s$.

\section{A Finitely Forced Model}
\label{finitely_forced}

We have shown that, in the interior of the phase space, entropy
maximizers with edge and $k$-star densities as constraints are multipodal. It is known from extremal graph theory that this is not true in general on the boundary of the phase space. We now briefly look at this issue using the concept of finitely forcible graphons introduced in~\cite{LS3}.

Let $h(x,y)$ be any doubly monotonic function with 0 as a regular value, and consider the graphon
\begin{equation}\label{EQ:Gen Tria} 
g(x,y) = \begin{cases} 1 & h(x,y)>0 \cr 0 & h(x,y) <0 \end{cases}.
\end{equation}
Then it is shown in~\cite{LS3} that for this graphon, the density of
the signed quadrilateral subgraph $Q$ (with signs going around the quadrilateral as $+$, $-$, $+$, $-$), is zero. In other words,
\begin{equation}\label{EQ:Quad Den} 
	t_Q(g) = \int g(w,x) [1-g(x,y)]g(y,z) [1-g(w,z)] dw\,dx\,dy\,dz =0,
\end{equation}
where we have labelled the four vertices as $w ,x, y, z$. It is straightforward to verify that $t_Q=t_1^2-2\tilde t_3+\tilde t_4$ where $\tilde t_3$ is the density of 3-chains while $\tilde t_4$ is the density of the un-signed quadrilateral:
\begin{equation}
\begin{array}{rcl}
	\tilde t_3(g) &=& \int g(w,x) g(x,y)g(y,z) dw\,dx\,dy\,dz, \\
        \tilde t_4(g) &=& \int g(w,x) g(x,y)g(y,z) g(z,w) dw\,dx\,dy\,dz .
\end{array}
\end{equation}

The triangle graphon $g_T$ defined by
\begin{equation}\label{EQ:Tri Graphon}
g_T(x,y) = \begin{cases} 1 & x+y > 1 \cr 0 & x+y < 1 \end{cases}
\end{equation}
is a special case of~\eqref{EQ:Gen Tria} with $h(x,y)=x+y-1$. For the triangle graphon we check that it has edge density $t_1=1/2$, 2-star density $t_2=1/3$, 3-chain density $\tilde t_3=5/24$ and quadrilateral density $\tilde t_4=1/6$. This clearly gives us $t_Q(g_T)=0$.

It is shown in~\cite{LS3} that {(up to rearranging vertices)} graphons of the form~\eqref{EQ:Gen Tria} are the only kind of graphon for which $t_Q(g)=0$. Moreover, among graphons of the form \eqref{EQ:Gen Tria} the density of edges minus the density of 2-stars is at most 1/6, and this upper bound is achieved uniquely by the graphon $g_T$. Thus the triangle graphon $g_T$ is finitely forcible with the two constraints:
\begin{enumerate}
\item $\zeta_1(g)\equiv t_1(g)-t_2(g)-\frac{1}{6}\ge 0$, i.e. the density of edges should be 1/6 greater than the density of 2-stars. 
\item $\zeta_2(g)\equiv t_Q(g)=t_1^2(g)-2\tilde t_3(g)+\tilde t_4(g)=0$, i.e. the density of the signed quadrilateral subgraph $Q$ should be zero.
\end{enumerate}

Here we look at a path toward the triangle graphon by considering the parameterized family of graphons $g_\alpha(x,y) = \alpha + (1-2\alpha) g_T(x,y)$ ($0\le\alpha\le 0.5$). We attempt to maximize the entropy among graphons that have the same values of ${t}=(t_1, t_2, \tilde t_3, \tilde t_4)$ as $g_\alpha$. We first check that 
\begin{eqnarray} 
t_1 & = & 1/2 \cr 
t_2 & = & (1-\alpha+\alpha^2)/3 \cr 
\tilde t_3 & = & (5-8\alpha+8\alpha^2)/24 \cr 
\tilde t_4 & = & (1-3\alpha+5\alpha^2-4\alpha^3+2\alpha^4)/6 .
\end{eqnarray}
This gives $\zeta_1(g_\alpha) =\alpha(1-\alpha)/3$ and $\zeta_2(g_\alpha) = \alpha(1+\alpha-4\alpha^2+2\alpha^3)/6$. 

We can show that by enforcing the densities $(t_1, t_2, \tilde t_3, \tilde t_4)$, we are approaching the triangle graphon from the interior of the profile when we let $\alpha\to 0$, as stated in the following theorem.
\begin{theorem} 
For any $\alpha>0$, the values of $t$ lie in the interior of the profile.  
\end{theorem}
\begin{proof} 
Since the graphon $g_a$ is strictly between 0 and 1, it is enough to show that the four functional derivatives, $\delta t_1/\delta g$, $\delta t_2/\delta g$, $\delta \tilde t_3/\delta g$ and $\delta \tilde t_4/\delta g$, are linearly independent functions of $x$ and $y$, since then by varying $g$ we can change ${t}$ in any direction to first order.  A simple computation shows that $\delta  t_1/\delta g(x,y)$ is a constant, $\delta t_2/\delta g(x,y)$ is a linear polynomial in $x$ and $y$, and $\delta \tilde t_3/\delta g(x,y)$ is a quadratic polynomial with an $xy$ term as well as linear terms. These three are analytic and manifestly linearly independent.

However, $\delta \tilde t_4/\delta g(x,y)$ is not analytic across the line $x+y=1$, and so cannot be a linear combination of the first three functional derivatives. To see this, it is enough to consider the case of $\alpha=0$.  $\delta t_4/\delta g(x,y)$ is a multiple of the probability of $x$ being connected to $y$ via a 3-chain $y$-$z$-$w$-$x$. When $x+y<1$, this is exactly $xy$, since if $z>1-y$ and $w>1-x$ then $z+w$ is automatically greater than $1$. When $x+y>1$, the requirement that $z+w>1$ provides an extra condition, and the functional derivative is strictly less than $xy$.
\end{proof}

We can now show that if we try to fit the density $t$ with $M$-podal graphons, then $M$ blows up as $\alpha$ goes to zero. More precisely,
\begin{theorem} For each positive integer $M$ there is an $\epsilon_M>0$
  such that for $\alpha <\epsilon_M$ there are no $M$-podal graphons whose
  densities ${t}$ are the same as those of $g_\alpha$.
\end{theorem}
\begin{proof} Suppose otherwise. Then we could find $M$-podal graphons
  for arbitrarily small $\alpha$. Since the space of $M$-podal (or smaller)
  graphons is compact, we can find a subsequence that converges to
  $g_0$ as $\alpha \to 0$. But densities vary continuously with the
  graphon, being simply integrals of products of $g$'s. So the
  densities of the $M$-podal graphon $g_0$ are the same as the
  densities of $g_T$. But that is a contradiction, since $g_T$ was
  finitely forced.
\end{proof}
Therefore, we know that $M \to \infty$ as $\alpha \to 0$. The fact
that the densities are polynomials in $g$ (of degree at most 4)
suggests that the growth should be at least a power law.

We performed two sets of numerical simulations. In the first set of simulations, we enforce the densities $(t_1, t_2, \tilde t_3, \tilde t_4)$ by solving the following minimization problem for some $\alpha$:
\begin{equation}
\max_{\{c_j\}_{1\le j\le K}, \{g_{i,j}\}_{1\le i,j\le K}} \S(g), \quad \mbox{subject to:}\quad t(g)=\boldsymbol (\tau_1,\tau_2,\tilde \tau_3,\tilde \tau_4), \quad \dsum_{1\le j\le K}c_j=1, \quad g_{ij}=g_{ji}.
\end{equation}
For values of $\alpha\in[0.001,0.5)$ we get multipodal maximizers
with a small number of podes. To be precise we obtain, numerically,
3-podal maximizers for $\alpha$ values in $(0.02, 0.5)$, 4-podal
maximizers for $\alpha$ values in $(0.004, 0.020)$, and 5-podal
maximizers for $\alpha$ values in $(0.001, 0.004)$. The transition
from 3-podal to 4-podal occurs around $\alpha=0.02$, and the
transition from 4-podal to 5-podal occurs around  $\alpha=0.004$; see the top row of Fig.~\ref{FIG:FinitelyForced} for typical 3-, 4- and 5-podal maximizers we obtained in this case.

In the second set of simulations, we solve a similar minimization problem that enforce the constaints on $\zeta_1$ and $\zeta_2$, instead of the four densities. For $\alpha$ values in $[0.001,0.5)$ we again get multipodal maximizers with a small number of podes. Precisely, we obtain 2-podal maximizers for $\alpha$ values in $(0.04, 0.5)$, 3-podal
maximizers for $\alpha$ values in $(0.015, 0.040)$, and 4-podal
maximizers for $\alpha$ values in $(0.001, 0.015)$. The transition
from 2-podal to 3-podal occurs around $\alpha=0.04$, and the
transition from 3-podal to 4-podal occurs around  $\alpha=0.015$; see the bottom row of Fig.~\ref{FIG:FinitelyForced} for typical 2-, 3- and 4-podal maximizers we obtained in this case.

Overall, our numerical simulations show that in a large fraction of the profile the maximizing graphons are multipodal with a small number of podes. The simulations also demonstrate that $M$ increases as $\alpha$ decreases. However, the numerical evidences are far from conclusive in the sense that we are not able to push $\alpha$ small enough to see the
(necessary) blow up behavior of $M$ more precisely, let alone the nature of the optimizing graphon as that occurs.
\begin{figure}[!ht]
\centering
\includegraphics[angle=0,width=0.25\textwidth]{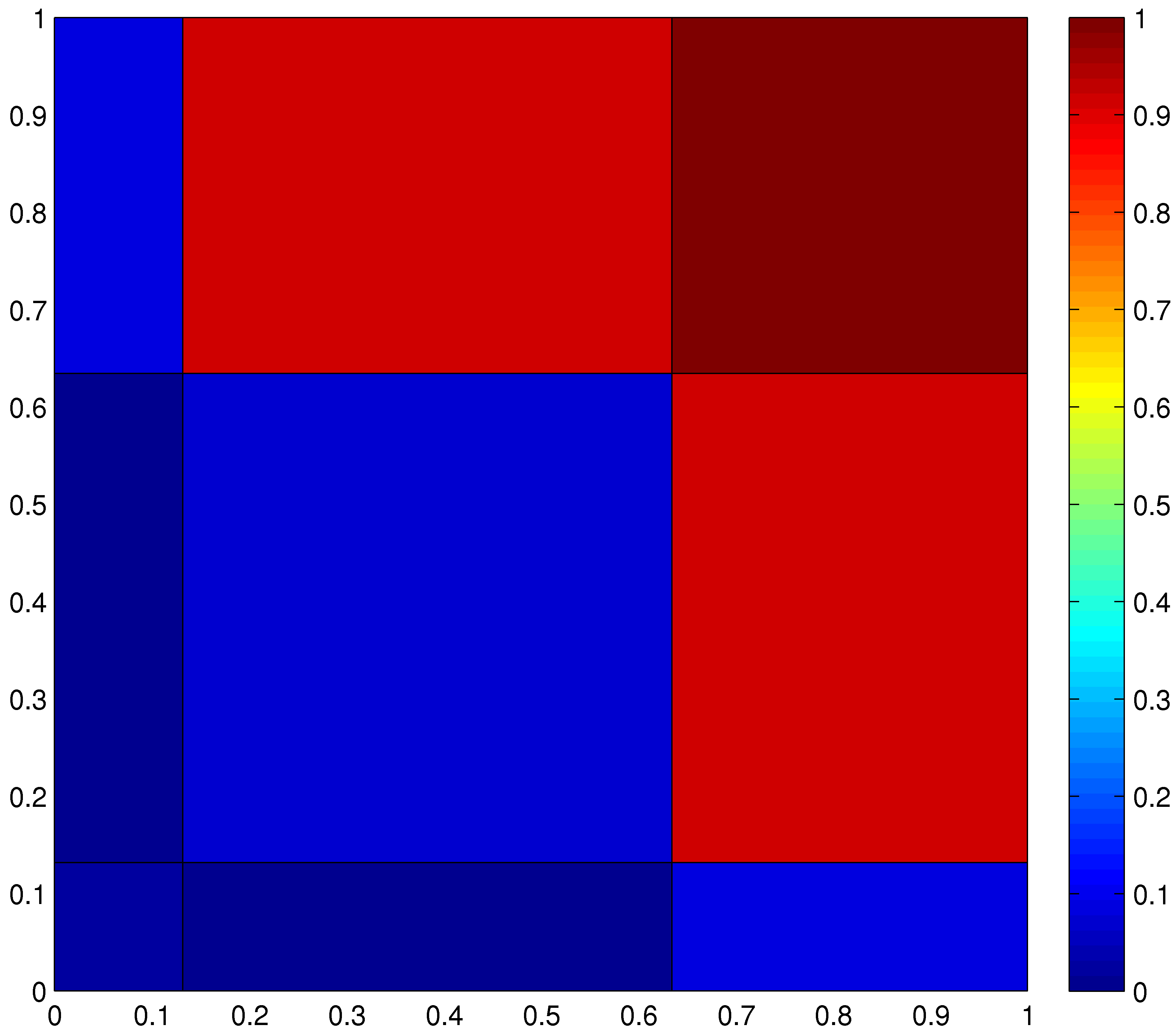}
\includegraphics[angle=0,width=0.25\textwidth]{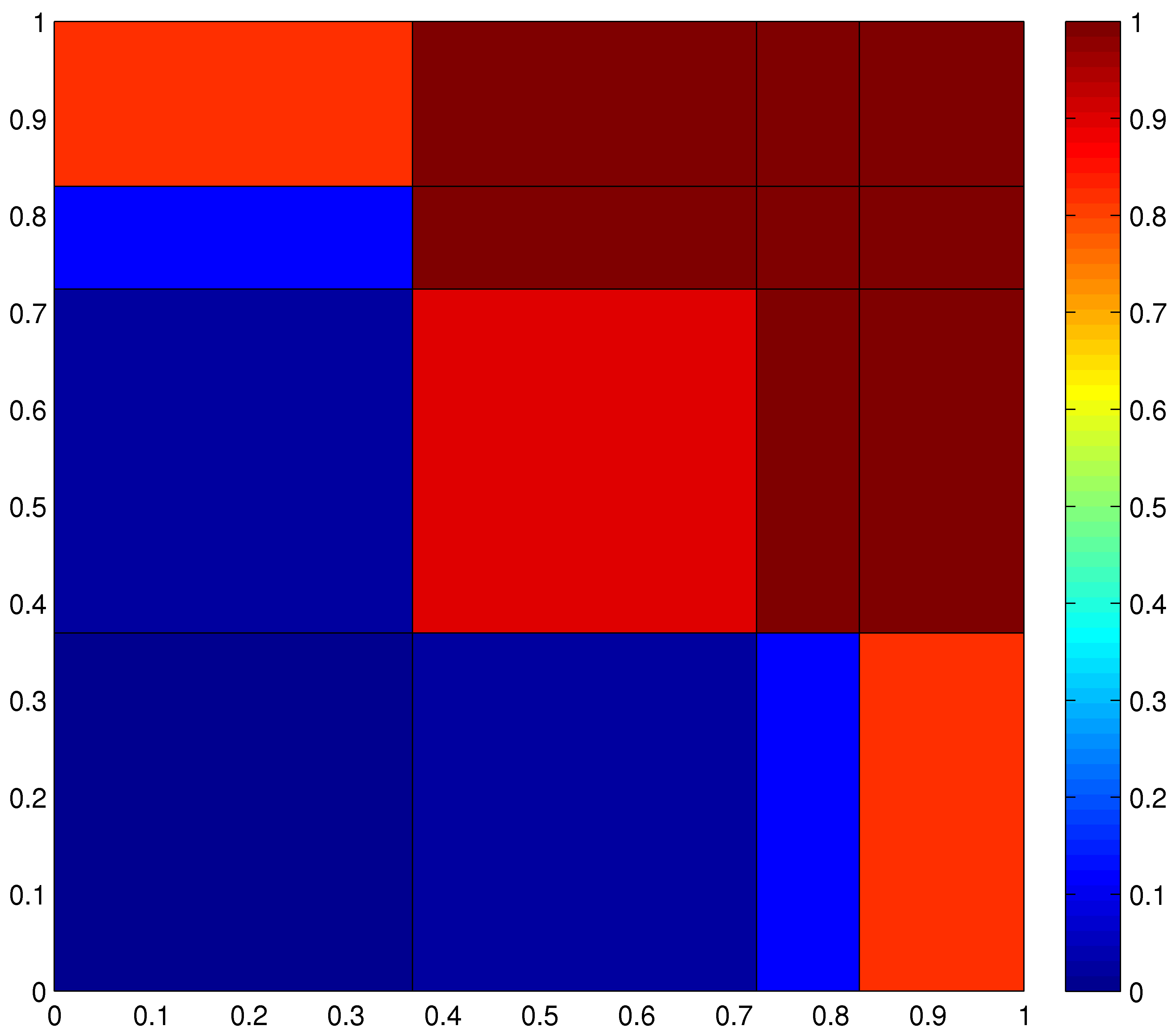}
\includegraphics[angle=0,width=0.25\textwidth]{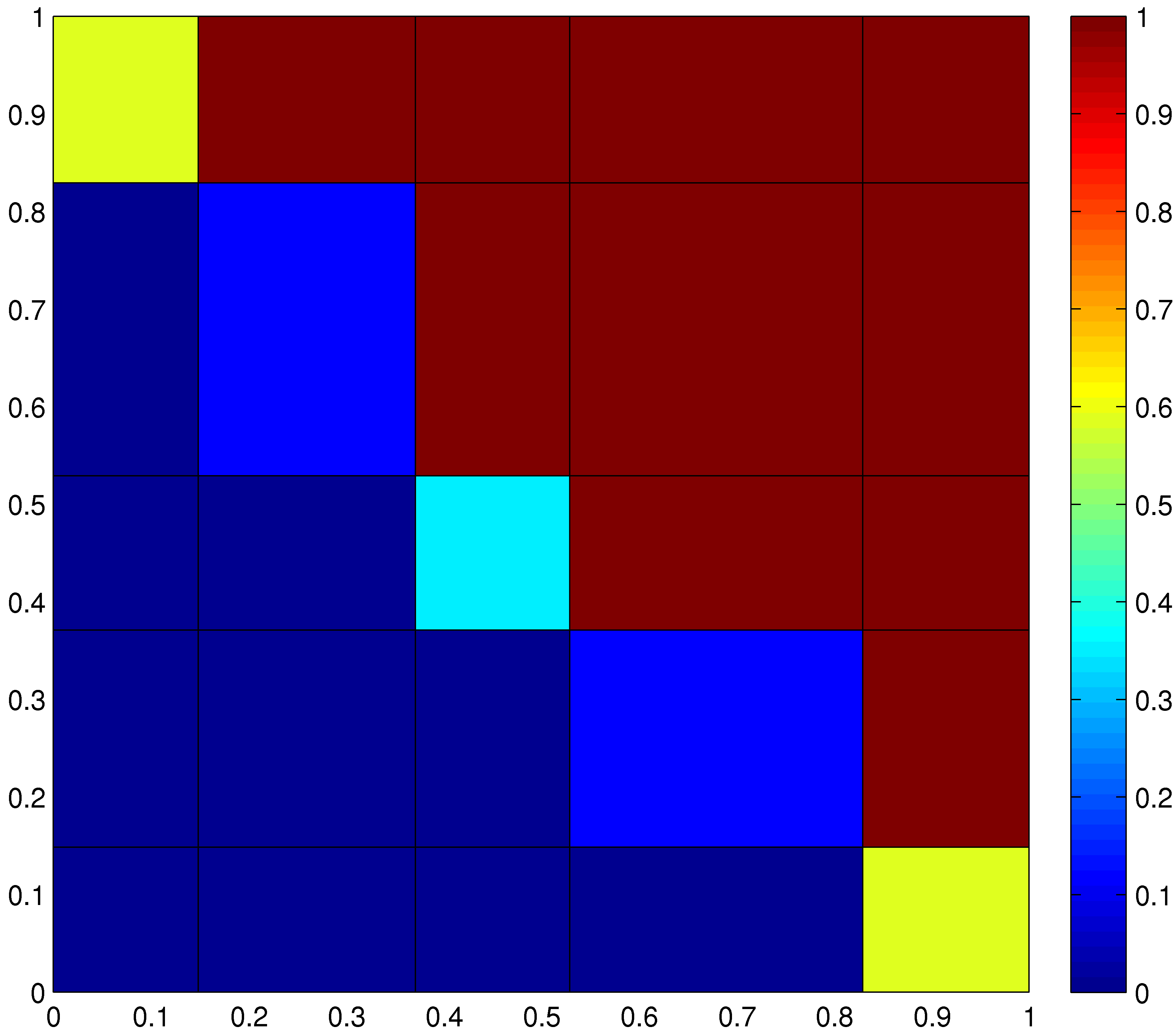}\\
\includegraphics[angle=0,width=0.25\textwidth]{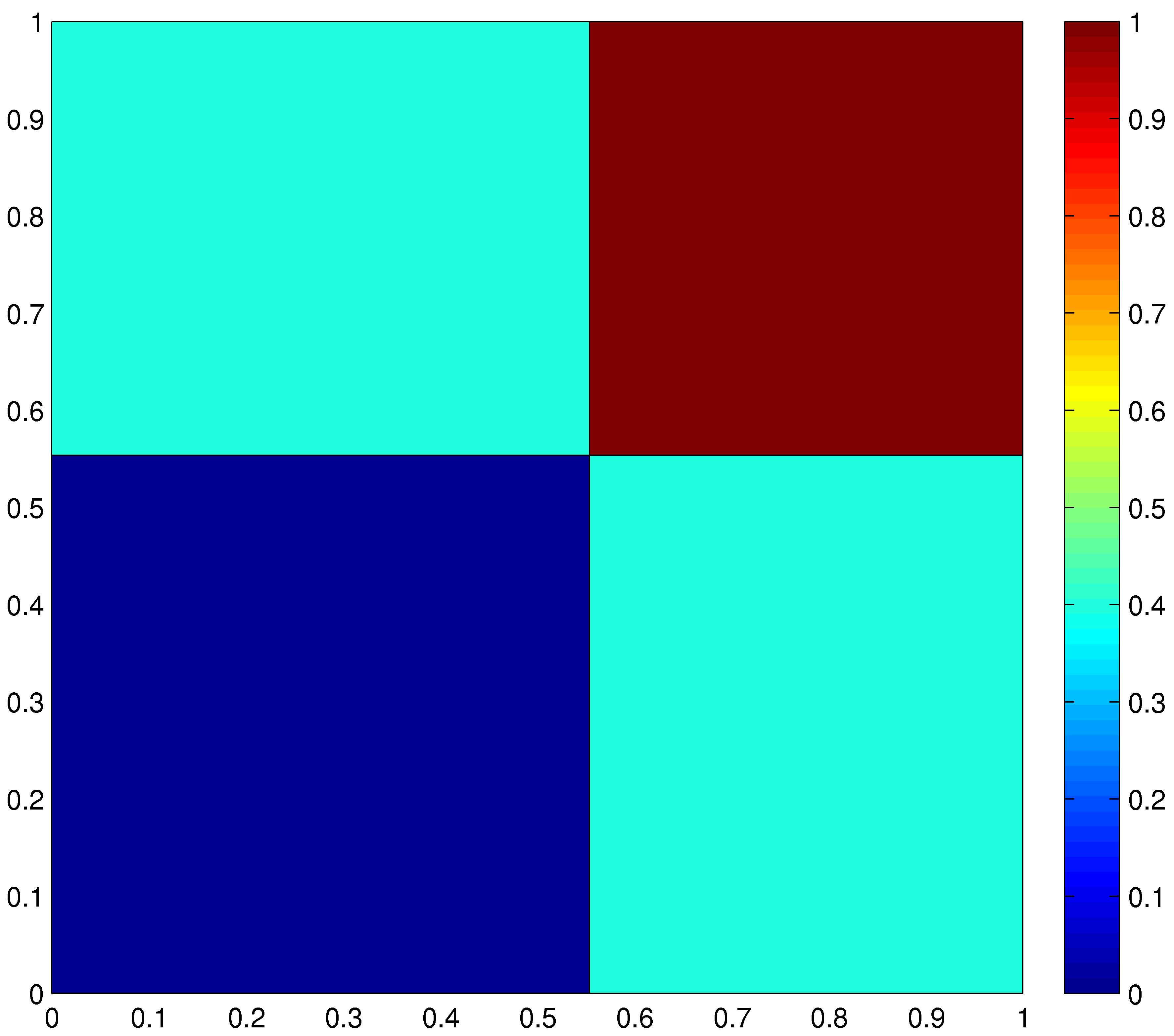}
\includegraphics[angle=0,width=0.25\textwidth]{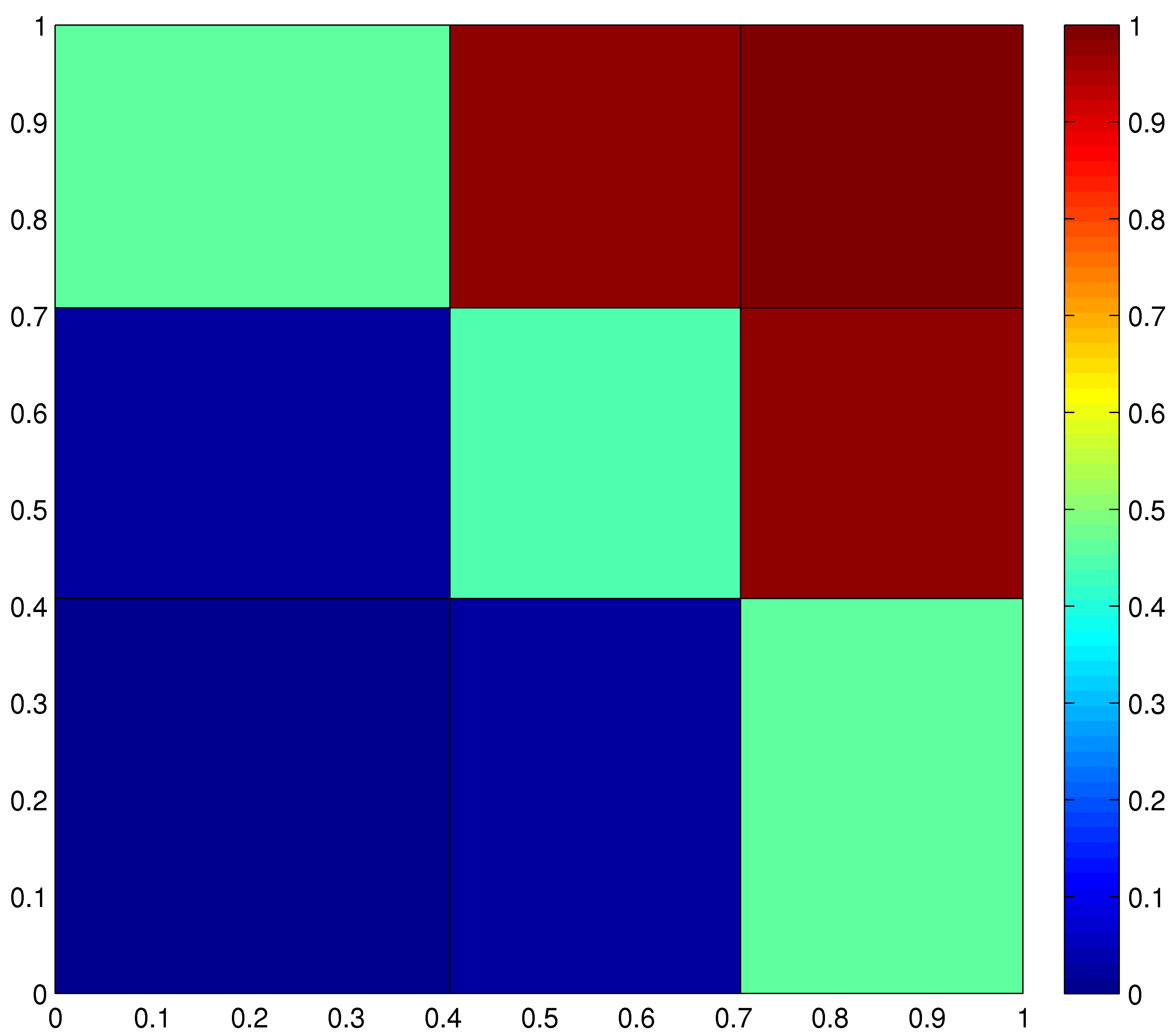}
\includegraphics[angle=0,width=0.25\textwidth]{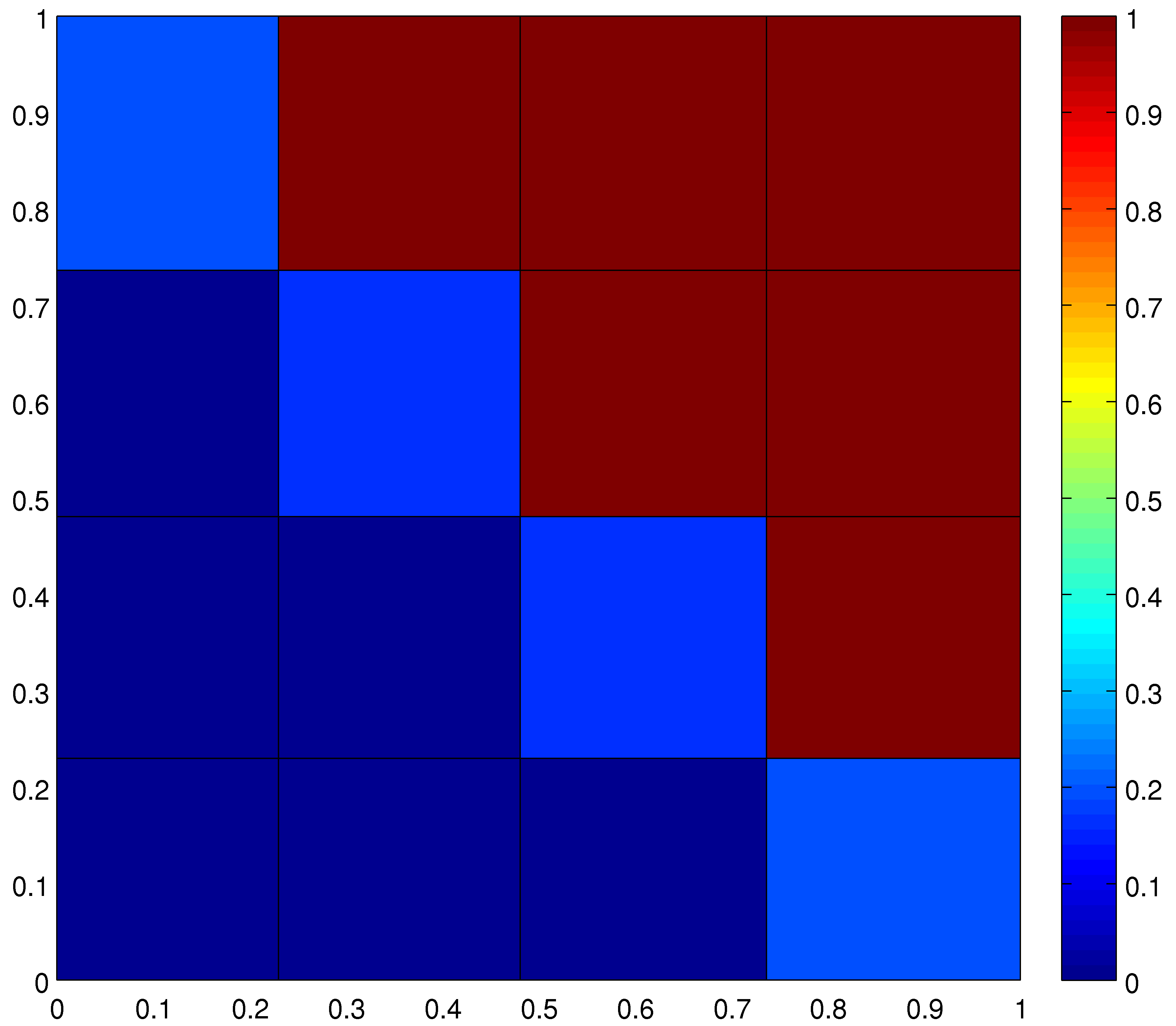}
\caption{Maximizing graphons for the finitely forced model. Top row: maximizers with $(t_1, t_2, \tilde t_3, \tilde t_4)$ constraints at $\alpha=0.0200$ (left, 3-podal), $\alpha=0.0050$ (middle, 4-podal) and $\alpha=0.0015$ (right, 5-podal); Bottom row: maximizers with $(\zeta_1, \zeta_2)$ constraints at $\alpha=0.0200$ (left, 2-podal), $\alpha=0.0050$ (middle, 3-podal) and $\alpha=0.0015$ (right, 4-podal).}
\label{FIG:FinitelyForced}
\end{figure}

\section{Conclusion}
\label{SEC:concl}

We first compare our results with exponential random graph models
(ERGMs), also based on given subgraph densities; see
\cite{CD,RY,AR2,LZ,Y,YRF,AZ} for previous mathematical work on their
asymptotics. For this we contrast
the basic optimization problems underlying ERGMs and the models of
this paper.

Intuitively the randomness in such random graph models arises, in
modeling large networks, by starting with an assumption that a
certain set of subgraphs $H=(H_1,\ldots,H_\ell)$ are `significant' for
the networks; one can then try to understand a large network as a
`typical' one for certain values $t_H(g)=(t_{H_1},\ldots,t_{H_\ell})$ of
the densities $t_{H_j}$ of those subgraphs.  Large deviations theory \cite{CV}
can then give
probabilistic descriptions of such typical graphs through a
variational principle for the constrained Shannon entropy,
$s_\T=\sup_{g|t_H(g)=\T} \S(g)$.

In this paper, as in \cite{RS1,RS2,RRS}, we use such constrained
optimization of entropy, and by analogy with statistical mechanics we
call such models `microcanonical'. In contrast, ERGMs
are analogues of `grand canonical' models of statistical
mechanics.  As noted in Section~\ref{SEC:notation},
our microcanonical version consists of maximizing $\S(g)$ over
graphons $g$ with fixed values $t_H(g)=\T$, leading to a 
constrained-maximum entropy $s_\T=\sup_{g|t_H(g)=\T}\S(g)$.  The optimizing graphons 
satisfy the
Euler-Lagrange variational equation, $\delta[\S(g)+\beta\cdot t_H(g)]=0$,
together with the constraints $t_H(g)=\T$, 
for some set of Lagrange multipliers $\beta=(\beta_1,\ldots,\beta_\ell)$. 

For the ERGM (grand canonical) approach, instead of fixing $t_H(g)$ one
\emph{maximizes} $\tilde F(g) =\S(g)+\beta\cdot t_H(g)$ for fixed $\beta$, obtaining
\begin{equation}
F_\beta=\sup_g \tilde F(g)=\sup_g[\S(g) +\beta\cdot t_H(g)].
\end{equation}

It is typical for there to be a loss of information in the 
grand canonical modelling of large graphs.
One way to see the loss is by comparing the
parameter (``phase'') space $\Sigma_{mc}=\{\T\}$ of the microcanonical
model with that for the grand canonical model,
$\Sigma_{gc}=\{\beta\}$. For each
point $\beta$ of $\Sigma_{gc}$ there are optimizing graphons $\tilde
g_{_{\beta}}$ such that $\tilde F(\tilde g_{_{\beta}})=F_{\beta}$, and for each
point $\T$ of $\Sigma_{mc}$ there are optimizing graphons $\tilde
g_{\T}$ such that $t_H(\tilde g_{\T})=\T$ and 
$\S(\tilde g_{\T})=s_{\T}$. Defining $\T'$ as $t_H(\tilde g_{_{\beta}})$ it follows
that $\tilde g_{_{\beta}}$ maximizes $\S(g)$ under some constraint $\T$, namely
$\S(\tilde g_{_{\beta}})=s_{\T'}$. But the converse fails: there are some
$\T$ for which no optimizing $\tilde g_{_{\beta}}$ satisfies $t_H(\tilde
g_{_{\beta}})=\T$ \cite{CD, RS1}. 

This asymmetry is particularly acute for the $k$-star
models we discuss in this paper: it follows from \cite{CD} that all of $\Sigma_{gc}$ is
represented only on the lower boundary curve of $\Sigma_{mc}$, $\Tk=\E^k$:
see Fig.~\ref{FIG:Phase-Boundary}. If one is
interested in the influence of certain subgraph densities 
in a large network it 
is therefore preferable to use constrained optimization of entropy rather than to
use the ERGM approach. 

Finally, in trying to understand the `phase transition' in the ERGM
edge/triangle model it seems significant that the functional
derivatives of the densities $\delta t/\delta g$ are linearly
dependent at the optimizing (constant) graphons relevant to that
transition. This was quite relevant in the perturbative analysis along
the Erd\H{o}s-R\'enyi curve in the microcanonical edge/triangle model
\cite{RS1}. And of course when the $\delta t/\delta g$ are linearly
dependent they cannot play their usual role as coefficients in the
expansion of the entropy $s_\T$. 

Next we consider the role of multipodal states in modeling large
graphs. In \cite{RS1,RS2,RRS} evidence, but not proof, of multipodal
entropy optimizers was found throughout the phase space of the
microcanonical edge/triangle model, and in this paper we have proven this to hold
throughout the phase space of all $k$-star models.  Consider more
general microcanonical graph models with constraints on edge density,
$e(g)$, and the densities $t_H(g)$ of a finite number of other
subgraphs, $H$. We are interested in the generality of multipodality
for entropy maximizing graphons in such models. As noted in the
Introduction there are known examples (in some sense `rare': see
Theorem 7.12 in \cite{LS3}) with nonmultipodality on phase space
boundaries, but this is not known to occur in the interior of any
phase space. 

To pursue this we first note
a superficial similarity between the subject of extremal
graphs and the older subject of `densest packings of bodies': Given a
finite collection $\C$ of bodies $\{B_1,\ldots,B_\ell\}$ in $\R^d$ determine
those nonoverlapping arrangements, of unlimited numbers of congruent
copies of the $B$'s, which maximize the fraction of $\R^d$ covered by
the $B$'s. (See \cite{Fej} for an overview.) As in extremal graph theory few examples have
been solved, the main ones being congruent spheres for dimensions $d\le 3$ and those
bodies which can tile space, such as congruent regular hexagons in the
plane. Based on this limited experience the assumption/expectation 
developed that for every
collection $\C$ there would be 
a `crystalline' densest packing, a packing whose
symmetry group was small (cocompact) in the group of symmetries of $\R^d$. 
This assumption was
proven incorrect in 1966 by the construction of `aperiodic tilings';
see \cite{Sen, Ra} for an overview.  One can therefore draw a parallel 
between aperiodic
`counterexamples' in the study of densest packings, and nonmultipodal
`counterexamples' in extremal graph theory.
Using nonoverlapping bodies to model molecules, physicists have applied
the formalism of statistical mechanics to packings of bodies. Packings
of spheres then give rise to the `hard sphere model' which is a simple
model for which simulation (not proof) shows the emergence of a
crystalline phase in the interior of the phase space \cite{Low}. More recently, aperiodic tilings have been
used to model quasicrystalline phases of matter. (See \cite{J} for an
introductory guide to quasicrystals.) Although it is
expected that the tilings, corresponding to optimal density on the
boundary of the microcanonical phase space, give rise to an emergent
quasicrystalline phase in the interior, there is much less simulation
evidence of this, as yet, than for crystalline phases emerging from
crystalline sphere packings; 
see \cite{AR1} and references
therein. 

Getting back to networks we note that
simple constraints give rise to multipodal optimal graphs on the
boundary of the phase space, and also \cite{RS1, RS2,RRS} multipodal
phases in the interior, in parallel to the crystalline situation in
packing. By analogy with packing therefore, a 
natural question is: do the nonmultipodal
`counterexamples' on the phase space boundary of random graph models give rise to
nonmultipodal \emph{phases} in the interior of the phase space, in parallel
to aperiodic tilings and quasicrystalline phases? (As in 
statistical mechanics a phase is defined as a connected open subset of the 
phase space of the model, in which the
entropy is analytic; see \cite{RS1}.)

\noindent {\bf Question 1.} Are random graph phases always multipodal?

Our attempt to investigate this in Section \ref{finitely_forced}
  was inconclusive.


Multipodality is a useful tool in
understanding phases. For instance in the edge/triangle model
[RS1, RS2, RRS] even a cursory inspection of the largest values of
such an optimizing graphon concentrates attention on the conditions under which
edges tend to clump together (fluid-like behavior) or push apart into
segregated patterns (solid-like behavior). 
More specifically, 
we note that in simulations of the edge/triangle model \cite{RRS} it
is very noticeable that at densities above the Erd\H{o}s-R\'enyi curve the
optimizing graphons are always monotone, while this is rarely if ever the case
below the curve. We proved in this paper that in $k$-star models, for which
densities are \emph{always} above the ER curve, the
optimizing graphons are always monotone. Consider a general model with
two densities, edges and some graph $H$. The ER curve is $\T_H=\E^k$ 
where $k$ is the number of edges in $H$. A natural question is:

\noindent {\bf Question 2.} Are the optimizing graphons always 
monotone above the ER curve in such random graph models?

In equilibrium statistical
mechanics \cite{Ru} one can rarely understand directly the equilibrium
distribution in a useful way, at least away from extreme values of
energy or pressure, so one determines the basic characteristics of a
model by estimating order parameters or other secondary quantities. In
random graph models multipodal structure of the optimizing state gives
hope for a more direct understanding of the emergent properties of
a model. This would be a significant shift of viewpoint. 
\section*{Acknowledgments}
The authors gratefully acknowledge useful discussions with Mei Yin
and references from Miki Simonovits, Oleg Pikhurko and Daniel Kr\'{a}l'.
The computational codes involved in this research were developed and
debugged on the computational cluster of the Mathematics Department of
UT Austin. The main computational results were obtained on the
computational facilities in the Texas Super Computing Center
(TACC). We gratefully acknowledge this computational support.
R. Kenyon was partially supported by the Simons Foundation.
This work was also partially supported by NSF grants DMS-1208191, DMS-1208941,
DMS-1321018 and DMS-1101326.



\begin{thebibliography}{AR}

\bibitem[AK]{AK} R. Ahlswede and G.O.H. Katona, Graphs with maximal number of adjacent
pairs of edges, Acta Math. Acad. Sci. Hungar. 32 (1978) 97-120

\bibitem[AR1]{AR1} D. Aristoff and C. Radin, First order phase transition 
in a model of quasicrystals, J. Phys. A: Math. Theor. 44(2011), 255001.

\bibitem[AR2]{AR2} D. Aristoff and C. Radin, Emergent structures in large 
networks, J. Appl. Probab. 50 (2013) 883-888. 

\bibitem[AZ]{AZ} D. Aristoff and L. Zhu, On the phase transition curve 
in a directed exponential random graph model, arxiv:1404.6514 

\bibitem[B]{B} {B. Bollobas}, Extremal graph theory, Dover
  Publications, New York, 2004.

\bibitem[BCL]{BCL} {C. Borgs, J. Chayes and L. Lov\'{a}sz}, Moments of
  two-variable functions and the uniqueness of graph limits,
  Geom. Funct. Anal. 19 (2010) 1597-1619.

\bibitem[BCLSV]{BCLSV} {C. Borgs, J. Chayes, L. Lov\'{a}sz, V.T. S\'os
    and K. Vesztergombi}, Convergent graph sequences I: subgraph
  frequencies, metric properties, and testing, { Adv. Math.} { 219}
  (2008) 1801-1851.

\bibitem[CD]{CD} S. Chatterjee and P. Diaconis, Estimating and understanding
exponential random graph models, Ann. Statist. 41 (2013) 2428-2461.

\bibitem[CDS]{CDS} S. Chatterjee, P. Diaconis and A. Sly, 
Random graphs with a given degree sequence, Ann. Appl. Probab. 21 (2011) 1400-1435.

\bibitem[CV]{CV} S. Chatterjee and S.R.S. Varadhan, The large deviation principle
for the Erd\H{o}s-R\'{e}nyi random graph, Eur. J. Comb. 32 (2011)
1000-1017.

\bibitem[Fej]{Fej} L. Fejes T\'oth, { Regular Figures},
  Macmillan, New York, 1964.


\bibitem[J]{J} C. Janot, {Quasicrystals: A primer},
Oxford University Press, Oxford, 1997.

\bibitem[Lov]{Lov} L. Lov\'asz, Large networks and graph limits,
American Mathematical Society, Providence, 2012. 

\bibitem[Low]{Low} H. L\"owen, Fun with hard spheres, In: ``Spatial Statistics and
Statistical Physics'', edited by K. Mecke and D. Stoyan, Springer
Lecture Notes in Physics, volume 554, pages 295--331, Berlin, 2000.

\bibitem[LS1]{LS1} {L. Lov\'{a}sz  and  B. Szegedy},
Limits of dense graph sequences, 
{ J. Combin. Theory Ser. B} { 98} (2006)  933-957.

\bibitem[LS2]{LS2} {L. Lov\'{a}sz  and  B. Szegedy},
Szemer\'edi's lemma for the analyst,
{ GAFA} { 17} (2007)  252-270.

\bibitem[LS3]{LS3} {L. Lov\'{a}sz  and  B. Szegedy},
Finitely forcible graphons,
{ J. Combin. Theory Ser. B} {101} (2011) 269-301.

\bibitem[LZ]{LZ}
E. Lubetzky and Y. Zhao,
On replica symmetry of large deviations in random graphs,
{Random Structures and Algorithms} {47} (2015) 109--146. 

\bibitem[N]{N} M.E.J. Newman, Networks: an Introduction, Oxford
  University Press, 2010.




\bibitem[P]{P} O. Pikhurko, private communication.

\bibitem[R]{R} C. Reiher, The clique density theorem, Ann. Math. (to
  appear), arXiv:1212.2454.

\bibitem[Ra]{Ra} C. Radin, Miles of Tiles, Student Mathematical Library,
Vol 1, Amer. Math. Soc., Providence, 1999.

\bibitem[RRS]{RRS} C. Radin, K. Ren and L. Sadun, The asymptotics of large constrained
graphs, J. Phys. A: Math. Theor. 47 (2014) 175001.

\bibitem[RS1]{RS1} C. Radin and L. Sadun, 
Phase transitions in a complex network, J. Phys. A: Math. Theor. 46 (2013) 305002.

\bibitem[RS2]{RS2}
C. Radin and L. Sadun, Singularities in the entropy of
  asymptotically large simple graphs, 
J. Stat. Phys. 158 (2015) 853--865. 

\bibitem[Ru]{Ru}
{D. Ruelle} 
{Statistical Mechanics; Rigorous Results},
Benjamin, New York, 1969.

\bibitem[RY]{RY} C. Radin and M. Yin, Phase transitions in exponential random
graphs, Ann. Appl. Probab. 23 (2013) 2458-2471.

\bibitem[Sen]{Sen} M. Senechal, {Quasicrystals and geometry}, Cambridge University
Press, Cambridge, 1995.

\bibitem[TET]{TET} H. Touchette, R.S. Ellis and B. Turkington, Physica
  A 340 (2004) 138-146.

\bibitem[Y]{Y} M. Yin, Critical phenomena in exponential random graphs,
J. Stat. Phys. 153 (2013) 1008-1021.

\bibitem[YRF]{YRF} M. Yin, A. Rinaldo and S. Fadnavis,
Asymptotic quantization of exponential random graphs, 
Ann. Appl. Probab. (to appear), arXiv:1311.1738 (2013).

\end{thebibliography}


\end{document}